\documentclass[11pt, twoside]{article}

\usepackage[latin1]{inputenc}
\usepackage{url, amsmath,enumerate,fancyhdr,amssymb, amsthm, pstricks, epsf, lscape, ifpdf, graphicx, float, mathtools}

\usepackage{blindtext}
\usepackage{relsize}

\usepackage{xcolor}

\usepackage{geometry}
\usepackage{setspace}
\numberwithin{equation}{section}

\usepackage{hyperref}
 \hypersetup{colorlinks,
 citecolor=blue,
 filecolor=blue,
 linkcolor=black,
 urlcolor=blue}
\mathchardef\mhyphen="2D

\newtheorem{theorem}{Theorem}[section]
\newtheorem{lemma}{Lemma}[section]

\theoremstyle{definition}
\newtheorem{definition}{Definition}[section]
\newtheorem{example}{Example}[section]

\newtheorem{proposition}{Proposition}[section]

\newtheorem{remark}{Remark}[section]

\makeatletter
\DeclareMathOperator*{\infd}{inf\vphantom{\operator@font p}} % inf with descender
\DeclareMathOperator*{\supx}{su\smash{\operator@font p}} % sup without descender
\makeatother

\usepackage[numbers]{natbib}
\usepackage{fancyhdr}

\begin{document}

\title{\textbf{On the Equivalence of Zero-Sum Games and Conic Programs}} 
\author{Nikos Dimou\footnote{Email: dimou@unc.edu.}}

\date{Department of Statistics and Operations Research\\University of North Carolina at Chapel Hill}

\maketitle

\pagestyle{myheadings}
\pagestyle{plain}
%\markboth{}{On the equivalence between the minimax theorem and strong duality...}

%\pagestyle{fancy}
%\fancyhf{}
%\fancyhfoffset[L]{1cm} % left extra length
%\fancyhfoffset[R]{1cm} % right extra length
%\rhead{\thepage}
%\lhead{\bfseries My name}
%\fancyhead[LO]{On the equivalence between the minimax theorem and strong duality for conic linear programming}
%\fancyhead[RE]{Nikos Dimou}
%\fancyhead[LE]{\thepage}
%\fancyhead[RO]{\thepage}
%\cfoot{\thepage}

\begin{abstract}
We prove the \textit{almost equivalence} of the minimax theorem and the strong duality theorem for a large class of games and conic programs. The previous fundamental results on the equivalence of linear programming and two-player zero-sum games with simplex-strategy sets are extended to Banach spaces, and a comprehensive framework unifying two-player zero-sum games and conic linear programs is established. Specifically, we show that for every zero-sum game with a bilinear payoff function and strategy sets that represent bases of convex cones, the minimax equality holds and its game value and Nash equilibria can be found by solving a primal-dual pair of conic programs. Conversely, the minimax theorem for the same class of games ``almost always" implies strong duality of conic linear programming. In fact, we give a game-dependent characterization of strict feasibility, and show that minimax is equivalent to a generalized version of Ville's theorem of the alternative. Several well-established game classes are embedded in the introduced model, including (i) semi-infinite, (ii) semidefinite, (iii) quantum, (iv) time-dependent, and (v) polynomial games, as well as (vi) the mixed extension of any continuous game with compact strategy sets.
\end{abstract}



\section{Introduction}\vspace{-0.2cm}

Consider a matrix $A\in\mathbb{R}^{m\times n}$ and the m--simplex $\Delta^{m}:= \left\{x\in\mathbb{R}^{m}:x\geq0,\;e_m^{\top}x=1\right\}$, where $e_m=(1,1,...,1)^{\top}\in\mathbb{R}^{m}$. The classical two-player zero-sum game \cite{VonNeumann1944-VONTOG-4} is an ordered triple $G=(\Delta^m,\Delta^n,u)$, where $\Delta^m$ ($\Delta^n$, respectively) is the set of mixed strategies for the row player $I$ (the column player $II$, resp.) and $u:\Delta^m\times\Delta^n\rightarrow\mathbb{R}$ is the payoff function of player $I$, given by $u(x,y):= x^{\top}Ay$. In his pioneering work \cite{Neumann1928}, von Neumann proved the following minimax theorem:
\begin{equation}\label{eq:1.1}
    \max_{x\in \Delta^m}\min_{y\in \Delta^n}\;x^{\top}Ay=v=\min_{y\in \Delta^n}\max_{x\in \Delta^m}\;x^{\top}Ay.
\end{equation}

Here, $v$ is called the \textit{value} of the game. It is well known (Dantzig \cite{dantzig2016linear}) that any such game has a value and optimal strategies which can be found by solving the following primal-dual pair of linear programs (LPs):\hypertarget{LPI}{}\hypertarget{LPII}{}\hypertarget{LPS}{}

\newpage
\vspace*{-0.9cm}
\begin{equation*}(P-LP)\;\;\;\;\;\;\;\;
\begin{array}{ll}
    \min\limits_{x}\;\; e_m^{\top}x
    \hspace{0.5cm}\text{subject to}\;\; x^{\top}A\geq e_n^{\top},\;\;\;x\geq 0
    \end{array}
\end{equation*}
\begin{equation*}\;\;(D-LP)\;\;\;\;\;\;\;\;
\begin{array}{ll}
    \max\limits_{y}\;\;e_n^{\top}y
    \hspace{0.5cm}\text{subject to}\;\; Ay\leq e_m,\;\;\;y\geq 0\;.\;\;\;
    \end{array}
    \end{equation*}
    
Following the use of the strong duality theorem of linear programming in the proof of the famous minimax theorem, Dantzig \cite{dantzig1951proof} considered the converse direction, and proved the ``almost equivalence" between the two fundamental theorems. The known ``gap" of the exact equivalence (in the sense that each theorem ``proves" -- according to \cite{von2024zero} -- the other), which was acknowledged by Dantzig, was addressed and filled several decades later by Adler \cite{adler2013equivalence}. In a recent paper, von Stengel \cite{von2024zero} gave an alternative proof of the fact that minimax proves strong duality of LP, and highlighted the importance of the Theorem of Tucker \cite{kuhn1956linear} in its completion.\par

Given these results the following question naturally arises: \textit{Can this relation between the two fundamental theorems be extended to a more generic class of games?} The only work we are aware of that addresses this question is that of Ickstadt et al. \cite{ickstadt2024semidefinite}, where Dantzig's aforestated results are generalized to the semidefinite setting. The latter question is strongly tied to the problem of finding exact, or even approximate, Nash equilibria in more complicated two-player zero-sum games, such as in games where both players have an infinite number of (pure) strategies. In fact, the need for extending the structural connection between the minimax theorem and strong duality to more general strategy spaces is precisely summarized in the following three concerns originally raised by Kuhn and Tucker in the preface of \cite{kuhn1953contributions}, regarding infinite (with respect to the number of strategies) and polynomial-like games \cite{dresher1950polynomial}, and which remain only partially resolved to this day:\quad\vspace{0.2cm}\par
 (5) \textit{To find structural theorems for the solutions of
comprehensive categories of infinite games, which go
beyond the polynomial-like games.}\vspace{-0.1cm}\par
 (6) \textit{How does one select the best from among the optimal
strategies in the case of infinite games?}\vspace{-0.1cm}\par
 (7) \textit{To find a computational technique of general applicability
for zero-sum two-person games with infinite
sets of strategies. A constructive method for
obtaining the optimal strategies for polynomial-like
games or some large class of non-trivial continuous
games would constitute a considerable contribution to
this problem.}

\vspace{-0.1cm}\subsection{Contributions}\vspace{-0.2cm}

In this paper, we prove the \textit{almost equivalence} between two-player zero-sum games and conic linear programs in reflexive Banach spaces. This represents an extension of the previous works of von Neumann \cite{VonNeumann1944-VONTOG-4}, Dantzig \cite{dantzig1951proof, dantzig2016linear}, Adler \cite{adler2013equivalence}, and von Stengel \cite{von2024zero}, who separately showed the equivalence in the case where the strategy spaces are standard simplices defined over $\mathbb{R}^n$, and the conic linear programs are LPs. The term ``almost" here is inspired by the one originally attributed to Dantzig \cite{dantzig1951proof}, but it has a different interpretation, in the sense that there is now no ``gap" on the equivalence to be filled. For a brief justification of this characterization and a terse description of our main contributions, we show that the strong duality theorem of conic linear programming always proves the minimax theorem for games with a bilinear payoff function and strategy sets that are bases of convex cones, while the converse is true in all but a particular game-dependent case.\par

In order to prove this concrete relation between the two theorems, we split the present work into two parts, each one dealing with one direction of the equivalence. In the first part, we address the open problems raised by Kuhn and Tucker \cite{kuhn1953contributions} by introducing  a wide class of two-player zero-sum games that can be described and solved by a primal-dual pair of conic linear programs. In particular, we show that in every two-player zero-sum game $G=(S,T,u)$ with a bilinear payoff function of the form $u(x,y)=\langle y,Ax\rangle$, for some linear operator $A$, and strategy sets $S,\;T$ that are \textit{cone-leveled}, the generalized minimax equation
\begin{equation}\label{eq:1.2}
    \sup_{x\in S}\inf_{y\in T}u(x,y)=v=\inf_{y\in T}\sup_{x\in S}u(x,y)
\end{equation} 
always holds and the game value $v$ is equal to $1/V$, where $V$ is the common optimal value of a primal-dual pair of conic programs, provided that a zero-duality gap exists for that specific pair. Informally, a cone-leveled set $S$ defined here is the intersection of a convex cone $C$ with a, possibly infinite, number of (parallel) hyperplanes. The primal-dual pair associated with the game, and whose common optimal value is equal to $V$, is of the form:
\hypertarget{primaldual}{} \hypertarget{primal}{} \hypertarget{dual}{}
\vspace*{-0.9 cm}\begin{center}
\begin{equation*}{(P)\;\;\;\;\;\;\;\;}
\begin{array}{ll}
    \inf\limits_{x}\;\; \langle c,x\rangle\\
    \hspace{0.05cm}\text{s.t.}\;\; Ax-b\in K^*\\ \;\;\;\;\;\;\;\;\;x\in C
    \end{array}
\end{equation*}
\end{center}
\vspace{-1.2 cm}\begin{center}
\begin{equation*}{\;\;\;\;\;\;(D)\;\;\;\;\;\;\;\;}
\begin{array}{ll}
    \sup\limits_{y}\;\; \langle y,b\rangle\\
    \hspace{0.05cm}\text{s.t.}\;\; -A^*y+c\in C^*\\ \;\;\;\;\;\;\;\;\;\hspace{0.035cm}y\in K
    \end{array}
    \end{equation*}
\end{center}
\vspace{-0.2cm}where $C,\;K$ are convex cones related to the geometric structure of the strategy sets $S,\;T$, respectively, $C^*,\;K^*$ are their positive duals, $A$ is the operator associated with the payoff function of the game, and $A^*$ is its adjoint operator.\par

Further, we show that in the special case where the strategy sets $S,\;T$ of the players are \textit{bases} of convex cones defined over reflexive Banach spaces (see Definition \ref{def:2.2} and \cite{casini2010cones}), strong duality of the -- associated with the game -- primal-dual pair always holds and it implies the minimax equation\vspace{0.2cm}
\begin{equation}\label{eq:1.3}
    \max_{x\in S}\min_{y\in T}\;\langle y,Ax\rangle=\min_{y\in T}\max_{x\in S}\;\langle y,Ax\rangle.
\end{equation}

Certain parts of the constructive proofs for the two minimax theorems given, namely the minimax theorem for cone-leveled sets (relation (\ref{eq:1.2}); Theorem \ref{th:3.1}) and the minimax theorem for bases (relation (\ref{eq:1.3}); Theorem \ref{th:3.3}), are standard and similar reduction methods of conic programs used here appear in the original works of von Neumann \cite{VonNeumann1944-VONTOG-4} and Dantzig \cite{dantzig1956constructive, dantzig2016linear}, as well as in \cite{ickstadt2024semidefinite}, in the context of linear programming and semidefinite programming, respectively. The difficulty of generalizing these methods to the infinite-dimensional case can be spotted in the far more general conic form of the problems \hyperlink{primal}{$(P)$} and \hyperlink{dual}{$(D)$}, in comparison to that of the linear programs \hyperlink{LPI}{$(P-LP)$} and \hyperlink{LPII}{$(D-LP)$}. In order to achieve the natural extension of such methods, we introduce a notion of ``equivalence" between two convex programs (Definition \ref{def:equivalence}), and then utilize the geometrical structure of cone-leveled sets (Lemma \ref{lem:3.4}) and properties of strictly positive functionals related to the interiors of closed convex cones in reflexive spaces \cite[Theorem 3.6]{casini2010cones}.\par

In the second main part of this paper, we show that the minimax theorem for bases of convex cones ``almost always" proves strong duality of conic linear programming, that is, we show that the two theorems are \textit{almost equivalent} (Theorem \ref{th:4.1}). For this direction of the equivalence, two major difficulties immediately arise when transitioning from $\mathbb{R}^n$ to any Banach space, both of which indicate that the methods previously used by both Adler \cite{adler2013equivalence} and von Stengel \cite{von2024zero} are no longer applicable: First and foremost, one now has to deal with abstract linear operators instead of real matrices. Therefore, a constructive proof of the Theorem of Tucker by Gordan's Theorem, as one was given by Adler \cite[Theorem 13]{adler2013equivalence} and a different one later on by von Stengel \cite[Theorem 6]{von2024zero} in the case of $\mathbb{R}^n$, is not possible in more general strategy spaces. Second, it is known that, given a primal-dual pair of conic linear programs other than LP, feasibility alone is often inadequate for strong duality to hold, even in finite-dimensional Euclidean spaces (see, e.g., \cite[Example 4.1.2]{wolkowicz2012handbook}). As a consequence, the famous Lemma of Farkas \cite[Proposition 1]{von2024zero}, which plays a significant role in the two aforementioned works, cannot guarantee a zero duality gap in the case where the problems \hyperlink{primal}{$(P)$} and \hyperlink{dual}{$(D)$} do not represent a dual pair of LPs, unless additional assumptions are made (Craven and Koliha \cite{craven1977generalizations}).\par

To overcome these issues, we take a different approach to the previous works \cite{adler2013equivalence,  dantzig1951proof, ickstadt2024semidefinite, von2024zero} and study two cases regarding the ``fairness" of any arbitrarily-chosen game ($v\neq0$ and $v=0$). First, we show that the existence of Nash equilibria is equivalent to a strong generalized version of Ville's theorem \cite{ville1938theorie} -- a theorem of the alternative -- when the game value is non-zero (Lemma \ref{lem:4.2}). This result directly implies strict feasibility, which is sufficient for strong duality between \hyperlink{primal}{$(P)$} and \hyperlink{dual}{$(D)$} to hold. We then give a complete characterization of strict feasibility for both \hyperlink{primal}{$(P)$} and \hyperlink{dual}{$(D)$} (Lemma \ref{lem:4.3}), when the game is fair (i.e., when the value of the chosen game is equal to zero). In contrast to Dantzig's initial attribution, the term ``almost" in the almost equivalence is then justified by the fact that there exists a particular case, when the game is fair, where neither of the problems \hyperlink{primal}{$(P)$} and \hyperlink{dual}{$(D)$} is strictly feasible, the complementary conditions fail, and a non-zero duality gap exists. It is shown that this unique case solely relies on the value of the game and on simple geometrical characteristics of the optimal strategies of the players. We then illustrate this ``pathology" via an example (Example \ref{ex:4.4}), which proves that the exact equivalence, as this was shown in \cite{adler2013equivalence, von2024zero}, isn't true in more general strategy spaces.\par

Moreover, we show that a theorem of the alternative, which can be understood as a natural generalization of Ville's Theorem \cite{ville1938theorie}, is equivalent to the minimax theorem (Theorem \ref{th:4.5}). This equivalence is important for two reasons. First, it implies that a different proof for Theorem \ref{th:3.3} can be given, one which does not use strong duality, but which instead relies on separating hyperplane theorems (see \cite{doi:10.1080/02331938908843428, yang2000theorems}). Second, it verifies the fact that the minimax theorem does not always prove strong duality: This theorem of the alternative does not always guarantee a zero duality gap for \hyperlink{primaldual}{$\{(P),(D)\}$}, in comparison to the Theorem of Ville which does so for a pair of LPs (see \cite{adler2013equivalence, von2024zero}).

\vspace{-0.1cm}\subsection{The importance of the almost equivalence}\vspace{-0.2cm}

The significance of the almost equivalence can be seen, in the simplest case possible, by the already known connection between the game $G=(\Delta^m,\Delta^n,u$) $\left(\text{with\;}u(x,y)=x^{\top}Ay\right)$ and the pair \hyperlink{LPS}{$\{(P-LP),(D-LP)\}$} of linear programs: The calculation of the game value and of a saddle point in mixed strategies reduces to the use of Dantzig's famous simplex method \cite{dantzig1956constructive, dantzig2016linear} or other suitable algorithms (e.g., \cite{karmarkar1984new}). Here, the primal-dual pair \hyperlink{primaldual}{$\{(P),(D)\}$} that is associated with the game constitutes a very general form of conic programs, whose subclasses include, among others, those of linear programming (LP), semidefinite programming (SDP), semi-infinite programming (SIP), quadratic programming (QP), copositive programming (CP), second order cone programming (SOCP), continuous linear programming (CLP), and general capacity (GCAP) problems. Therefore, the toilsome task of computing a game value and (approximate) Nash equilibria simply boils down to applying appropriate algorithms, a plethora of which are already known for all of the preceding classes of convex programs; see \cite{alizadeh2003second,anderson1989capacity,  dantzig2016linear, karmarkar1984new, lai1992extremal, levinson1966class, nesterov1994interior, nocedal1999numerical, shapiro2009semi, weiss2008simplex, wolkowicz2012handbook} and the references within.\par
The applicability of the almost equivalence is, in fact, twofold. Besides the practical utility of Theorem \ref{th:3.3} in the computation of Nash equilibria, the second main result, Theorem \ref{th:4.1}, equips conic program solvers with a novel method for determining strict feasibility for primal-dual pairs: One simply has to study the value and optimal strategies of their favourite constructed two-player zero-sum game in an appropriately indicated manner. By ``construction" of a game here we mean the selection of interior points $\alpha\in\text{int}C^*$ and $\beta\in\text{int}K^*$ that generate the strategy sets. Therefore, the fundamental problem of verifying \textit{Slater's CQ} -- a vital condition for both the existence of a zero-duality gap \cite{anderson1983review, bonnans2013perturbation, nash1987linear, shapiro2001duality} and the fast convergence of interior-point algorithms to approximate stationary points \cite{nesterov1994interior, nocedal1999numerical, wolkowicz2012handbook} of conic optimization problems -- is precisely equivalent to studying the properties of suitably chosen zero-sum games (see Example \ref{ex:4.4} for a simple illustration in the case of SDPs). To the best of our knowledge, Theorem \ref{th:4.1} is the first result that directly relates the existence of strictly feasible solutions with characteristics of zero-sum games (here the game value and geometric properties of Nash equilibria) for such a large class of conic programs.

\vspace{-0.1cm}\subsection{Related work}\label{sec:relatedwork}\vspace{-0.2cm}

A number of independent works have given partial answers to the questions raised by Kuhn and Tucker \cite{kuhn1953contributions}, by considering very specific classes of zero-sum games and conic linear programs. In 1989, Anderson et al. \cite{anderson1989capacity} discussed the application of general capacity problems in two-player zero-sum games and the existence and calculation of Nash equilibria in mixed strategies. References regarding the connection between games and CLPs have appeared previously in Bellman \cite{bellman1966dynamic} and Nash and Anderson \cite{nash1987linear}. Using Duffin's \cite{duffin1956infinite} fundamental work on infinite programs, Soyster \cite{soyster1975semi} studied zero-sum games where one of the players has a set of infinite pure strategies. The same class of games was studied a few years later by Tijs \cite{tijs1979semi}, who utilized the theory of semi-infinite programming to prove the minimax theorem and the existence of optimal strategies for one of the players. Another area of mathematical programming related to various classes of zero-sum games is that of semidefinite programming. In particular, Parrilo \cite{parrilo2006polynomial} showed that
polynomial games (games with a polynomial payoff function; see \cite{dresher1950polynomial}) can be described and solved by a pair of SDP problems. SDPs also play a significant role in non-interactive quantum games; see, for example, Jain and Watrous \cite{jain2009parallel}, and the references within. A recent treatment of generalized semidefinite games through semidefinite programming has been developed by Ickstadt et al. \cite{ickstadt2024semidefinite}.\par
In Section \ref{sec:5} it is shown that, in all of the above games, the strategy sets of the players have the cone-leveled structure introduced. Therefore, all of the aforestated classes of two-player zero-sum games are encompassed within the proposed game model, and the corresponding computational and existential results regarding Nash equilibria are derived from invoking Theorems \ref{th:3.1} and \ref{th:3.3}.\par

The minimax theorems of the present work are not the first results that guarantee the existence of Nash equilibria in games defined over such a broad topological setting. In fact, Theorem \ref{th:3.3} is a special case of the classical minimax theorem for games with continuous payoff functions and convex compact strategy sets in locally convex topological vector spaces (Fan \cite{fan1952fixed}, Glicksberg \cite{glicksberg1952further}). This is due to the fact that bases of convex cones are compact in reflexive Banach spaces \cite[Lemma 3.4]{casini2010cones}. Note, however, that the proof of Theorem \ref{th:3.3} is constructive and -- in comparison to the previous classical works -- it does not rely on fixed-point theorems, as we construct pairs of conic linear programs and relate their optimal solutions with Nash equilibria.\par

Regarding the equivalence of the minimax theorem and strong duality, we have already mentioned the works of Dantzig \cite{dantzig1951proof}, Adler \cite{adler2013equivalence}, and von Stengel \cite{von2024zero}. Both Adler \cite[Theorem 11]{adler2013equivalence} and von Stengel \cite[Proposition 4]{von2024zero} have also shown that the Theorem of Ville \cite{ville1938theorie} and the classical minimax theorem (\ref{eq:1.1}) are, in fact, equivalent. Therefore, Theorem \ref{th:4.5} represents a natural extension of these results to reflexive Banach spaces.\par

The most relevant work to ours is that of Ickstadt et al. \cite{ickstadt2024semidefinite}. There, the authors present a natural generalization of Dantzig's \cite{dantzig1951proof} result of the equivalence to the space of real symmetric matrices. The corresponding optimization setting is that of semidefinite programming, a setting where strong duality can fail (see \cite[Example 4.1.2]{wolkowicz2012handbook}), so the notion of \textit{almost equivalence} in \cite{ickstadt2024semidefinite} is aligned with the one presented here. Therefore, our work can be regarded as an extension of the latter paper rather than one of \cite{adler2013equivalence, dantzig2016linear, VonNeumann1944-VONTOG-4,von2024zero}. At the same time, the two results of the almost equivalence, namely \cite[Theorem 5.3]{ickstadt2024semidefinite} and Theorem \ref{th:4.1}, are of different nature, as are the methods used for their proofs. In the following points we highlight these key differences and argue that, besides the evident wider range of game spaces covered, the almost equivalence presented here is stronger than that of \cite{ickstadt2024semidefinite} (and thus than that of \cite{dantzig1951proof}):\vspace{-0.2cm}
\begin{enumerate}
    \item The application of the minimax theorem appears to be different in the two works. In that of Ickstadt et al. \cite{ickstadt2024semidefinite}, a symmetric version of minimax is invoked to prove the existence of a solution to a generalized Dantzig's game \cite{dantzig1951proof}, with game value equal to zero. In turn, this is used to prove the consistency of a semidefinite system. Instead, in the present work we invoke the minimax theorem to guarantee the existence of a game value (not necessarily zero) and of Nash equilibria (not necessarily symmetric) for any chosen game, which we then study in order to determine the existence of feasible interior points.\vspace{-0.1cm}
     \item Another major difference can be spotted in the restrictive nature of the two results. In order to provide a criterion for strong duality, the authors of \cite{ickstadt2024semidefinite} follow Dantzig's method of studying a specific symmetric game. Rather than utilizing a fixed game, Theorem \ref{th:4.1} allows one to freely choose any strategy sets $S,\;T$, and then study the resulting two-player zero-sum game.\vspace{-0.1cm}
    \item In the cases where the main result of \cite{ickstadt2024semidefinite} guarantees strong duality for a pair of SDPs, it does so without providing any information about the existence of strictly feasible solutions. Here, we are able to guarantee strict feasibility in a number of game-dependent cases, with the latter being a significantly stronger property than that of strong duality (see Proposition \ref{prop:2.1}).\vspace{-0.1cm}
    \item Unlike duality theory of linear programming, it is common in more generic conic program settings, such as in SDPs, that a zero-duality gap exists while at least one program does not attain its optimal value \cite[Example 4.1.1]{wolkowicz2012handbook}. This work covers this ordinary case, in comparison to \cite{ickstadt2024semidefinite} which does not  (Example \ref{ex:example2}).\vspace{-0.1cm}
    \item Theorem \ref{th:4.1} can guarantee a zero-duality gap -- in fact strict feasibility -- in a certain pathological subcase not covered by the main result of \cite{ickstadt2024semidefinite} (Theorem \ref{th:SDPsubcase}).
\end{enumerate}

It should be mentioned that the criterion of \cite{ickstadt2024semidefinite} is oftentimes easier to verify in practice; the latter demands solving a certain feasibility system, while Theorem \ref{th:4.1} requires solving a user-chosen pair of conic programs. This should not be surprising, however, as Theorem \ref{th:4.1} guarantees additional properties and covers more cases in a larger span of game spaces and optimization settings.\par

A final, noteworthy mention is the work of Karlin \cite{karlin1953theory}, who was the first to study two-player zero-sum games defined over Banach spaces, with basic cone-leveled strategy sets, and payoff functions of the more general form $u(x,y):=\frac{\langle y,Ax\rangle}{\langle y,Bx\rangle}$. However, no connections with duality theory of conic convex programming appear in his work.

\vspace{-0.1cm}\subsection{Structure of the paper}\label{sec:1.3}\vspace{-0.2cm}

The rest of the paper is organized as follows. In Section \ref{sec:2} we give a self-contained review of two-player zero-sum games and duality theory of conic linear programming, and we formally define cone-leveled sets and bases of convex cones. The first direction of the almost equivalence is proved in Section \ref{sec:3}, while Section \ref{sec:4} includes the opposite one. Section \ref{sec:5} contains a number of applications and examples to various classes of two-player zero-sum games. In Section \ref{sec:6} we extend Theorem \ref{th:3.1} to games with finite intersections and unions of cone-leveled sets. In the last section we briefly discuss future directions and set some open questions related to our work.

\vspace{-0.1cm}\section{Preliminaries and notation}\label{sec:2}\vspace{-0.2cm}

Let $(X,W)$, $(Z,Y)$ be dual pairs of vector spaces, meaning that bilinear forms $\langle\cdot,\cdot\rangle_X:W\times X\rightarrow\mathbb{R}$, $\langle\cdot,\cdot\rangle_Z:Y\times Z\rightarrow\mathbb{R}$ are defined, and that each vector space separates the points of the other with respect to the corresponding bilinear form, i.e., if $\langle w,\cdot\rangle_X=0$ then $w=0_W$, and if $\langle\cdot,x\rangle_X=0$ then $x=0_X$. Assuming that there is no confusion between the pairings, we denote both bilinear functions by $\langle\cdot,\cdot\rangle$. In order to deal with the most general case possible, for the first part of Section \ref{sec:3} the pairings are assumed to be dual pairs of locally convex topological vector spaces, while for the second part these will represent pairs of reflexive Banach spaces equipped with their strong topologies (see Subsection \ref{sec:3.2}). Recall that a Banach space $X$ is reflexive if and only if every bounded sequence in $X$
has a weakly convergent subsequence. Also recall that every finite-dimensional normed space is a reflexive Banach space. Given a linear operator $A:X\rightarrow Z$ we assume that \textit{for every $y\in Y$ there exists a unique $w\in W$ such that $\langle y,Ax\rangle=\langle w,x\rangle$ for every $x\in X$}. That is, the adjoint of $A$, $A^*:Y\rightarrow W$, is a linear operator and it is well defined by $\langle A^*y,x\rangle=\langle y,Ax\rangle$.  The latter assumption particularly holds if we equip the spaces $X,\;Z$ with any compatible topologies with respect to their pairing.\par

 Our primary goals rely heavily on dealing with convex cones. A set $C\subseteq X$ is a \textit{convex cone} if $\lambda C\subseteq C$ for every $\lambda\geq0$ and $\kappa C+(1-\kappa)C\subseteq C$ for every $\kappa\in[0,1]$. A cone with a non-empty interior is called \textit{solid}. We denote the interior of a cone $C$ by int$C$. The \textit{positive dual cone} of $C$ is defined by
\begin{equation*}
    C^*:= \bigl\{w\in W:\langle w,x\rangle \geq0\;\forall x\in C\bigr\}.
\end{equation*}
 We also define the set of all strictly positive functionals associated with the cone $C$, called the \textit{strict positive dual cone} of $C$, as follows:
\begin{equation*}
    C^{*s}:= \Bigl\{w\in W:\langle w,x\rangle>0\;\forall x\in C\setminus\{0_X\}\Bigr\}.
\end{equation*}

\vspace{-0.3cm}\subsection{Conic linear programming}\label{sec:2.1}\hypertarget{primalbeta}{}\hypertarget{dualalpha}{}\hypertarget{newprimaldual}{}\vspace{-0.2cm}

Consider the primal program \hyperlink{primal}{$(P)$} and its dual \hyperlink{dual}{$(D)$}, as above. We will oftentimes refer to this pair as \hyperlink{primaldual}{$\{(P),(D)\}$}. When dealing with primal-dual pairs of this form, where $b$ is replaced by some $\beta\in Z$ and $c$ by some $\alpha\in W$, we shall refer to them as \hyperlink{newprimaldual}{$\{(P_\beta),(D_\alpha)\}$}. The feasible (or feasibility) sets of \hyperlink{primal}{$(P)$} and \hyperlink{dual}{$(D)$}, respectively, are defined by
\begin{equation*}
    \mathcal{F}(P):= \{x\in X:x\in C,\;Ax-b\in K^*\}\mbox{\;\;and\;\;}\mathcal{F}(D):= \{y\in Y: y\in K,\;-A^*y+c\in C^*\}.  
\end{equation*}
 The optimal values of \hyperlink{primal}{$(P)$} and \hyperlink{dual}{$(D)$}, are given by
\begin{equation*}
   \text{val}(P):= \inf\{\langle c,x\rangle :x\in\mathcal{F}(P)\}\mbox{\;\;and\;\;} \text{val}(D):= \sup\{\langle y,b\rangle : y\in\mathcal{F}(D)\},
\end{equation*}
when the corresponding feasibility sets are non-empty (we define $\text{val}(P):=  +\infty$, $\text{val}(D):=  -\infty$ when $\mathcal{F}(P)=\emptyset$, $\mathcal{F}(D)=\emptyset$, respectively). In the sequent, we say that the pair \hyperlink{primaldual}{$\{(P),(D)\}$} is \textit{feasible}, if both problems have non-empty feasible sets. Furthermore, we say that $x^*$ is an \textit{optimal solution} for \hyperlink{primal}{$(P)$} if $x\in\mathcal{F}(P)$ and $\langle c,x^*\rangle=\text{val}(P)$. Similarly, we say that $y^*$ is an \textit{optimal solution} for \hyperlink{dual}{$(D)$} if $y^*\in\mathcal{F}(D)$ and $\langle y^*,b\rangle=\text{val}(D)$. We denote by $\mathcal{S}(P)$ and $\mathcal{S}(D)$ the sets of optimal solutions for \hyperlink{primal}{$(P)$} and \hyperlink{dual}{$(D)$}, respectively.\par

 The following weak duality relation always holds: $\text{val}(P)\geq \text{val}(D)$. We say that there exists (or is) a \textit{zero (or no) duality gap} between \hyperlink{primal}{\hyperlink{primal}{$(P)$}} and \hyperlink{dual}{\hyperlink{dual}{$(D)$}} if and only if $\text{val}(P)=\text{val}(D)$. If, in addition, $\mathcal{S}(D)$ is non-empty, then we say that \textit{strong duality holds} between \hyperlink{primal}{$(P)$} and \hyperlink{dual}{$(D)$}. Note that this relation is not symmetric.  Moreover, there exists a zero duality gap between \hyperlink{primal}{\hyperlink{primal}{$(P)$}} and \hyperlink{dual}{\hyperlink{dual}{$(D)$}} and the feasible solutions $x^*,\;y^*$ are optimal if and only if the following condition (\textit{complementary slackness}) holds:
\begin{equation}\label{eq:2.3}
    \langle y^*, Ax^*-b\rangle=\langle c-A^*y^*,x^*\rangle=0.
\end{equation}
A plethora of conditions, also known as \textit{constraint qualifications} (CQs), under which strong duality for the pair \hyperlink{primaldual}{$\{(P),(D)\}$} holds, exist in the relevant literature of conic linear programming, whether this takes a special form  \cite{alizadeh2003second,anderson1989capacity,  dantzig2016linear, karmarkar1984new, lai1992extremal, levinson1966class, nesterov1994interior, shapiro2009semi, weiss2008simplex, wolkowicz2012handbook} or not \cite{anderson1983review, bonnans2013perturbation, khanh2019necessary, kretschmer1961programmes, nash1987linear, shapiro2001duality}. Perhaps the most common CQ that appears in the literature is the one related to the existence of feasible interior points, also known as \textit{Slater's CQ}:
\begin{definition}\label{def:slaters}
    We say that $x\in X$ is a \textit{strictly feasible solution} for \hyperlink{primal}{$(P)$} if $x\in C$ and $Ax-b\in \text{int}K^*$. Accordingly, we say that $y\in Y$ is a \textit{strictly feasible solution} for \hyperlink{dual}{$(D)$} if $y\in K$ and $-A^*y+c\in \text{int}C^*$. Moreover, we say that \hyperlink{primal}{$(P)$} (\hyperlink{dual}{$(D)$} resp.) is \textit{strictly feasible} if it has a strictly feasible solution.
\end{definition}
The following strong duality result (Bonnans and Shapiro \cite[Theorem 2.187]{bonnans2013perturbation}, Shapiro \cite[Proposition 2.8]{shapiro2001duality}, Nash and Anderson \cite[Theorem 3.13]{nash1987linear}, Anderson \cite[Theorem 10]{anderson1983review}) will be used extensively in order to prove the almost equivalence. It is assumed here that $C,\;K$ are closed convex cones, $K^*$ is a solid cone, $X,\;Y$ are Banach spaces endowed with their strong topologies and $W=X^*,\;Z=Y^*$ are their corresponding dual spaces endowed with their weak-$*$ (or strong if they are all reflexive spaces) topologies.\quad\vspace{0.2cm}
\begin{proposition}\label{prop:2.1}
    \textit{If \hyperlink{primal}{$(P)$} is strictly feasible and \hyperlink{dual}{$(D)$} is feasible, then $\text{val}(P)=\text{val}(D)$ and the set $\mathcal{S}(D)$ is non-empty, convex, and bounded.}
\end{proposition}

\begin{remark}\label{rem:reflexive}
    The main reason why we will work with pairs of reflexive Banach spaces of the form $(X,X^*)$ to prove the almost equivalence is because $X$ is isometric with its double dual $X^{**}$. This allows for various important duality results, such as Proposition \ref{prop:2.1}, to be properly ``dualized", i.e., to be stated for the corresponding dual problem without any additional assumptions or complicated modifications in their proofs. In essence, this follows from the fact that, if the reflexive spaces $X,\;X^*$ are equipped with their corresponding strong topologies ($Y$, $Y^*$, resp.) and the cones $C,\;K$ are closed, then the dual of the dual problem \hyperlink{dual}{$(D)$} is exactly the primal problem \hyperlink{primal}{$(P)$} (see Bonnans and Shapiro \cite[pp. 126--128]{bonnans2013perturbation}). Therefore, in what follows we will be stating and invoking the dual counterparts of certain results (e.g., Proposition \ref{prop:2.1}) without providing a complete, formal proof. An additional reason why reflexive Banach spaces are suitable for the game-theoretic framework presented can be spotted in the resulting relationship between the interiors of $C^*$ and $C^{*s}$, for any closed convex cone $C\subseteq X$ (Lemma \ref{lem:3.4}).   
\end{remark}

\vspace{-0.3cm}\subsection{Two-player zero-sum games}\vspace{-0.2cm}

 A \textit{two-player zero-sum game} is an ordered triple $G=(S, T,u)$, where $S\subseteq X$ and $T\subseteq Y$ are the \textit{strategy sets} of players $I$ and $II$, respectively, and $u:S\times T\rightarrow \mathbb{R}$ is the \textit{payoff function} of player $I$ (the payoff function of player $II$ is given by $-u$). We say that the \textit{minimax equality holds} if
\begin{equation}\label{eq:2.4}
    \max_{x\in S}\min_{y\in T}u(x,y)=\min_{y\in T}\max_{x\in S}u(x,y).
\end{equation}
 We define the \textit{security levels} of players $I$ and $II$ by $\underline{v}$ and $\overline{v}$, respectively, where $\underline{v}$ ($\overline{v}$, resp.) is the left-hand-side (rhs, resp.) of (\ref{eq:1.2}). These are also called the \textit{lower value} and \textit{upper value} of the game, respectively. The following minimax relation is derived from definition:
\begin{equation}\label{eq:2.5}
    \underline{v}\leq\overline{v}\;.
\end{equation}
  We say that the game $G=(S,T,u)$ has a \textit{value} if equation (\ref{eq:1.2}) holds, and we denote that value by $v:=  \underline{v}=\overline{v}$. Further, we say that $(x^*,y^*)\in S\times T$ is a \textit{Nash equilibrium} or a \textit{saddle point} if and only if
\begin{equation}\label{eq:2.6}
    u(x^*,y)\geq u(x^*,y^*)\geq u(x,y^*)\;\;\forall x\in S,\;\forall y\in T.
\end{equation}
 A strategy $x^*$ of player $I$ is called \textit{optimal} if it satisfies $u(x^*,y)\geq v$ $\forall y\in T$, or equivalently, if $v=\inf\limits_{y\in T}u(x^*,y)$, when $v$ exists. Accordingly, a strategy $y^*$ of player $II$ is called \textit{optimal} if it satisfies $u(x,y^*)\leq v$ $\forall x\in S$, or equivalently, if 
 \;$v=\sup\limits_{x\in S}u(x,y^*)$, when $v$ exists. (Several authors, e.g., von Stengel \cite{von2002computing}, also call $x^*$ a \textit{best response} to $y^*\in T$, if $u(x^*,y^*)\geq u(x,y^*)$ $\forall x\in S$.) Observe that the strategies $x^*$,\;$y^*$ of players $I$ and $II$ are optimal if and only if $(x^*,y^*)$ is a saddle point. We denote the sets of optimal strategies for players $I$ and $II$ by $\mathcal{B}^{I}$ and $\mathcal{B}^{II}$, respectively.

\vspace{-0.1cm}\subsection{Cone--leveled sets and bases of convex cones}\vspace{-0.2cm}

\begin{definition}\label{def:2.2}
A set $S\subseteq X$ is called \textit{cone-leveled} if there exist $\alpha\in W\setminus\{0_W\}$, a convex cone $C\subseteq X$, and a non-empty set $H\subseteq [p,q]$ with $p,q\in H$ for some real values $0<p\leq q$, such that 
\begin{equation}\label{eq:2.7}
    S=\bigl\{x\in C:\langle \alpha,x\rangle\in H\bigr\},
\end{equation}
and such that the set $S_r:=  \bigl\{x\in C:\langle \alpha,x\rangle =r\bigr\}$ is non-empty for $r=p,q$.  If $p=q$, then $H=\{p\}$ and $S$ is said to be a \textit{basic cone-leveled} set.
\end{definition}
A basic cone-leveled set S defined by a closed convex cone $C\subseteq X$ and some $\alpha\in \text{int}C^*$ is called a \textit{base} of $C$ (see, e.g., \cite{casini2010cones}).\par

Basic cone-leveled sets can be thought of as generalizations of conic slices (or conic sections) from the real Euclidean space to more generic topological vector spaces. Further, if $H=[p,q]$, then $S$ is convex; it is in fact the intersection of a convex cone with two closed half spaces.

\vspace{-0.1cm}\section{Strong Duality proves the Minimax Theorem}\label{sec:3}\vspace{-0.2cm}

In this section we prove the key minimax theorems of this paper. We first describe the class of games considered. Let $A:X\rightarrow Z$ be a linear operator, where $X,Z$ represent locally convex topological vector spaces, as in Section \ref{sec:2}. We study two-player zero-sum games $G=(S,T,u)$ of the following form:\hypertarget{f1}{}\hypertarget{f2}{}\hypertarget{f3}{} \\ \\
\textbf{\hyperlink{f1}{(F1)}}: The strategy sets $S\subseteq X$, $T\subseteq Y$ are cone-leveled sets.\vspace{0.1cm} \\ 
\textbf{\hyperlink{f2}{(F2)}}: The payoff function is given by $u(x,y):= \langle y,Ax\rangle$, for every $x\in S$, $y\in T$.\vspace{0.1cm} \\
\textbf{\hyperlink{f3}{(F3)}}: The security levels $\underline{v},\;\overline{v}$ of players $I$ and $II$, respectively, take finite values over $S,\;T$.\vspace{0.2cm}

 The last property indicates that we only deal with essential games, in which no player can guarantee an infinite payoff, and the game value is not infinite. The cone-leveled set $S$ here is defined as in (\ref{eq:2.7}). Moreover, the cone-leveled strategy set $T$ is defined by a convex cone $K\subseteq Y$, a set $Q\subseteq[p',q']$ \big{(}with $p',q'\in Q$ for some $p',q'\in (0,+\infty)$\big{)} and some $\beta\in Z$, in similar vein to (\ref{eq:2.7}). Throughout we will assume, for simplicity of exposition, that the sets $Q$, $H$ are subsets of the intervals $[1,\delta]$ and $[\gamma,1]$ for some $\gamma\in(0,1]$ and $\delta\geq1$, respectively.

 \vspace{-0.1cm}\subsection{Games with cone-leveled strategy sets}\label{sec:3.1}\vspace{-0.2cm}

\begin{theorem}[Minimax theorem for cone-leveled sets]\label{th:3.1}
Consider a two-player zero-sum game $G=(S,T,u)$ defined by \hyperlink{f1}{(F1)}-\hyperlink{f3}{(F3)}, and assume that $\underline{v}>0$. If the pair \hyperlink{newprimaldual}{$\{(P_\beta),(D_\alpha)\}$} is feasible and $\text{val}(P_\beta)=\text{val}(D_\alpha)>0$, then $\underline{v}=\overline{v}$ and $v=\frac{1}{\text{val}(P_\beta)}$. Moreover, if $x^*$ is an optimal solution for \hyperlink{primalbeta}{$(P_\beta)$} and $y^*$ is an optimal solution for \hyperlink{dualalpha}{$(D_\alpha)$}, then $(vx^*,vy^*)$ is a saddle point.
\end{theorem}

Before we come to the proof of Theorem \ref{th:3.1}, we show that the lower and upper values of a game that satisfies \hyperlink{f1}{(F1)}--\hyperlink{f3}{(F3)} can be expressed as the optimal values of two convex optimization problems. This technical result is standard when proving that strong duality implies the minimax equality in less general conic settings, such as in linear programming \cite[pp. 286]{dantzig2016linear} and semidefinite programming \cite[Lemmas 4.4, 4.5]{ickstadt2024semidefinite}.

\begin{lemma}\label{lem:securitylevels}
    Consider a two-player zero-sum game $G=(S,T,u)$ defined by \hyperlink{f1}{(F1)}-\hyperlink{f3}{(F3)}, and assume that $\underline{v}>0$. Further, consider the convex optimization programs \hypertarget{p1}{}
\begin{center}
\vspace*{-0.9cm}\begin{equation*}{(P1)\;\;\;\;\;\;\;\;}
\begin{array}{ll}
    \sup\limits_{\xi,x}\;\; \xi\\
     \textup{s.t.}\;\;\langle y,Ax\rangle\geq\xi\;\;\forall y\in T\\ \;\;\;\;\;\;\;\xi>0,\;x\in S
    \end{array}
\end{equation*}
\end{center}
\begin{center}\hypertarget{d1}{}
\vspace{-0.7cm}\begin{equation*}{\;(D1)\;\;\;\;\;\;\;\;\;}
\begin{array}{ll}
    \inf\limits_{\zeta,y}\;\; \zeta\\
     \textup{s.t.}\;\; \langle A^*y,x\rangle\leq\zeta\;\;\forall x\in S \\ \;\;\;\;\;\;\;\zeta>0,\;y\in T
    \end{array}
\end{equation*}
\end{center}
Then, $\underline{v}=\text{val}(P1)$ and $\overline{v}=\text{val}(D1)$. Moreover, if $\underline{v}=\overline{v}$, then $\mathcal{B}^{I}=\mathcal{S}(P1)$ and $\mathcal{B}^{II}=\mathcal{S}(D1)$.
\end{lemma}

\begin{proof}
We first show that $\text{val}(P1)=\underline{v}$. Define the set $\widetilde{S}:= \left\{x\in S:\inf\limits_{y\in T}\langle y,Ax\rangle >0\right\}$. This is non-empty because $\underline{v}>0$. Further, for every $x\in\widetilde{S}$, the pair $(\xi,x)$ with $\xi=\xi(x):= \inf\limits_{y\in T}\langle y,Ax\rangle$ represents a feasible solution for \hyperlink{p1}{$(P1)$}. Therefore, we get 
\begin{equation}\label{eq:9}    \underline{v}=\sup\limits_{x\in S}\inf\limits_{y\in T}\langle y,Ax\rangle=\sup\limits_{x\in \widetilde{S}}\inf\limits_{y\in T}\langle y,Ax\rangle=\text{val}(P1).
\end{equation}  
For the second equality, we define the set $\widetilde{T}:= \left\{y\in T:\sup\limits_{x\in S}\langle y,Ax\rangle>0\right\}$.Then, for every $y\in\widetilde{T}$, the pair $(\zeta,y)$ constitutes a feasible solution for \hyperlink{d1}{$(D1)$}, where $\zeta=\zeta(y):= \sup\limits_{x\in S}\langle y,Ax\rangle$. The bound $\overline{v}>0$ yields $\widetilde{T}=T$, and we therefore get the relation
 \begin{equation}\label{eq:10}
    \overline{v}=\inf\limits_{y\in T}\sup\limits_{x\in S}\langle y,Ax\rangle=\inf\limits_{y\in \widetilde{T}}\sup\limits_{x\in S}\langle y,Ax\rangle=\text{val}(D1).
 \end{equation}
Suppose now that $\underline{v}=\overline{v}$. Then, the value of the game $v$ exists. By definition and (\ref{eq:9}), $x^*\in S$ is an optimal strategy for player $I$ if and only if it is an optimal solution for \hyperlink{p1}{$(P1)$}. Similarly, (\ref{eq:10}) implies that $y^*\in T$ is an optimal strategy for player $II$ if and only if it is an optimal solution for \hyperlink{d1}{$(D1)$}.
\end{proof}

A natural step towards proving Theorem \ref{th:3.1} would be to show that \hyperlink{p1}{$(P1)$} and \hyperlink{d1}{$(D1)$} are dual to each other, and that their optimal values are equal. While it is relatively straightforward to show this duality in the linear programming case \cite[pp. 286, Theorem 1]{dantzig2016linear}, the latter task is not trivial in settings of arbitrary optimization complexity, not even in simple special cases. For instance, in the semidefinite setting the authors of \cite{ickstadt2024semidefinite} prove duality between \hyperlink{p1}{$(P1)$} and \hyperlink{d1}{$(D1)$} by constructing a specific block diagonal matrix and using the suitable structure of the linear operator $A$ (a tensor in that setup).\par

Despite the fact that we deal with a far more generic class of games where the payoff operator $A$ can have any linear structure, we are still able to guarantee a zero-duality gap between the programs of Lemma \ref{lem:securitylevels} by a series of reformulations, equivalences and reductions of convex programs, instead of involved constructions. To this end, we introduce a notion of ``equivalence" between two optimization problems:

\begin{definition}\label{def:equivalence}
    We say that two optimization problems $(OP1)$ and $(OP2)$ defined over spaces $X_1$ and $X_2$, respectively, are \textit{equivalent} if there exist functions $\varphi_1:X_1\to X_2$, $\varphi_2:X_2\to X_1$ and a bijection $f:L_1\to L_2$ for some non-empty $L_1,L_2\subseteq\mathbb{R}$ such that \vspace{-0.3cm}
    \begin{enumerate}[(1)]
        \item $\varphi_1\left(\mathcal{F}(OP1)\right)\subseteq\mathcal{F}(OP2)$,\vspace{-0.2cm}
        \item $\varphi_2\left(\mathcal{F}(OP2)\right)\subseteq\mathcal{F}(OP1)$, \vspace{-0.2cm}
        \item $\text{val}(OP2)=f\left(\text{val}(OP1)\right)$, whenever $\mathcal{F}(OP1)\neq\emptyset$ and $\text{val}(OP1)$ is finite.
    \end{enumerate}
\end{definition}

Observe that if $(OP1)$ and $(OP2)$ are equivalent, then $\mathcal{F}(OP1)\neq\emptyset$ if and only if $\mathcal{F}(OP2)\neq\emptyset$. This observation along with the bijective property of $f$ imply that Definition \ref{def:equivalence} is symmetric. We will refer to $\varphi_2$ as the \textit{inverse} of $\varphi_1$ (and vice-versa).\par

We now show that two certain convex programs are equivalent according to the above definition. Let $\alpha_1,\alpha_2,...,\alpha_m\in W$ and a convex cone $C\subseteq X$. We define the sets $\mathcal{D}_i:=  \left\{x'\in C:\langle \alpha_i,x'\rangle\leq 0\right\}$ for $i\in\{1,2,...,m\}$. In turn, for some index $i_0\in\{1,2,...,m\}$, we consider the programs
\hypertarget{phat}{}
\vspace*{-0.9cm}\begin{center}
\begin{equation*}{(\widehat{P}\hspace{0.05cm})\;\;\;\;\;\;\;\;}
\begin{array}{ll}
    \sup\limits_{\xi,x}\;\; \xi\\
    \hspace{0.05cm}\text{s.t.}\;\;\langle y,Ax\rangle\geq\xi\;\;\forall y\in T\\
\;\;\;\;\;\;\;\hspace{0.04cm}\langle\alpha_i,x\rangle=1\;\;\forall i\in\{1,2,...,m\}\\
    \;\;\;\;\;\;\;\;\xi>0, \;x\in C
    \end{array}
\end{equation*}
\end{center}
\vspace{-0.9 cm}\begin{center} \hypertarget{phati0}{}
\begin{equation*}{\;\;\;\;\;\;\;\;\;\;\;\;\;\;\;\;\;(\widehat{P}_{i_0})\;\;\;\;\;\;\;}
\begin{array}{ll}
    \inf\limits_{x'}\;\;\;\langle \alpha_{i_0},x'\rangle\\
    \hspace{0.05cm}\text{s.t.}\;\; \langle y,Ax'\rangle\geq 1\;\;\forall y\in T\\ \;\;\;\;\;\;\;\hspace{0.04cm}\langle\alpha_i,x'\rangle=\langle\alpha_j,x'\rangle\;\;\forall i\neq j \in \{1,2,...,m\}\\
    \;\;\;\;\;\;\;\;x'\in C\setminus \left(\displaystyle\bigcap_{i=1}^{m}\mathcal{D}_i\right)
    \end{array}\vspace{0.2cm}
\end{equation*}
\end{center}

\begin{lemma}\label{lem:3.2}
    The programs \hyperlink{phat}{$(\widehat{P}\hspace{0.05cm})$} and \hyperlink{phati0}{$(\widehat{P}_{i_0})$} are equivalent.
\end{lemma}

\begin{proof}
The programs \hyperlink{phat}{$(\widehat{P}\hspace{0.05cm})$}, \hyperlink{phati0}{$(\widehat{P}_{i_0})$} are defined over the spaces $(0,+\infty)\times X$ and $X$, respectively. If $(\xi,x)\in\mathcal{F}(\widehat{P}\hspace{0.05cm})$, then $\frac{x}{\xi}\in\mathcal{F}(\widehat{P}_{i_0})$. Conversely, if $x'\in\mathcal{F}(\widehat{P}_{i_0})$, then $\left(\frac{1}{\langle\alpha_{i_0},x'\rangle},\frac{x'}{\langle\alpha_{i_0},x'\rangle}\right)\in\mathcal{F}(\widehat{P}\hspace{0.05cm})$. Therefore, the functions $\varphi_1:(0,+\infty)\times X\to X$ with $\varphi_1(\xi,x)=\frac{x}{\xi}$ and $\varphi_2:X\to (0,+\infty)\times X$ with $\varphi_2(x')=\frac{1}{\langle \alpha_{i_0},x'\rangle}\left(1,x'\right)$ satisfy properties (1) and (2) of Definition \ref{def:equivalence}. In turn, the third property of Definition \ref{def:equivalence} is satisfied by the bijection $f:(0,+\infty)\to(0,+\infty)$ with $f(x)=1/x$. To see this, assume that \hyperlink{phat}{$(\widehat{P}\hspace{0.05cm})$} is feasible and that $\text{val}(\widehat{P}\hspace{0.05cm})<+\infty$. Let $x'$ feasible for \hyperlink{phati0}{$(\widehat{P}_{i_0})$} and take $(x,\xi)$ feasible for \hyperlink{phat}{$(\widehat{P}\hspace{0.05cm})$}, as given by the inverse of $\varphi_1$. Then $\xi\leq \text{val}(\widehat{P}\hspace{0.05cm})$. Since $x'$ was arbitrary, we derive $\frac{1}{\text{val}(\widehat{P}\hspace{0.05cm})}\leq \text{val}(\widehat{P}_{i_0})$. Similarly, if $(x,\xi)$ is feasible for \hyperlink{phat}{$(\widehat{P}\hspace{0.05cm})$}, then $x':=  \frac{x}{\xi}$ is feasible for \hyperlink{phati0}{$(\widehat{P}_{i_0})$} and the relation $\text{val}(\widehat{P}_{i_0})\leq\langle\alpha_{i_0},x'\rangle$ holds. The inequality $\frac{1}{\text{val}(\widehat{P}\hspace{0.05cm})}\geq \text{val}(\widehat{P}_{i_0})$ follows.
\end{proof}

It is next argued that cone-leveled sets are not more general than basic cone-leveled sets in games with a bilinear payoff function.

\begin{lemma}\label{lem:basic}
 Consider a two-player zero-sum game $G=(S,T,u)$ defined by \hyperlink{f1}{(F1)}-\hyperlink{f3}{(F3)}. Then $\underline{v}=\sup\limits_{x\in S_1}\inf\limits_{y\in T}\langle y,Ax\rangle$ and $\overline{v}=\inf\limits_{y\in T_1}\sup\limits_{x\in S}\langle y,Ax\rangle$.
\end{lemma}

\begin{proof}
    We only prove the first relation (the second is analogously proved). It suffices to show that all strategies of player $I$ that do not lie in the 1-hyperplane are dominated. To see this, let $x\in C$ with $\langle\alpha,x\rangle=\gamma'\in[\gamma,1)$. Then, $\hat{x}:= \frac{x}{\gamma'}\in S_1=\{x\in C:\langle \alpha,x\rangle=1\}$ and $\langle y,A\hat{x}\rangle=\frac{1}{\gamma'}\langle y,Ax\rangle>\langle y,Ax\rangle$ for every $y\in T$. Hence, the supremum of $\inf\limits_{y\in T}\langle A^*y,\cdot\rangle$ is attained over $S_1$.
\end{proof}

\begin{proof}[Proof of Theorem 1]
By Lemmas \ref{lem:securitylevels} and \ref{lem:basic}, \hyperlink{p1}{$(P1)$} has the same optimal value and optimal solutions with the program
\hypertarget{p2}{}
\begin{center}
\vspace{-0.9cm}\begin{equation*}{(P2)\;\;\;\;\;\;\;\;}
\begin{array}{ll}
    \sup\limits_{\xi,x}\;\; \xi\\
    \hspace{0.05cm}\text{s.t.}\;\;\langle y,Ax\rangle\geq\xi\;\;\forall y\in T\\
\;\;\;\;\;\;\;\hspace{0.04cm}\langle\alpha,x\rangle=1\\
    \;\;\;\;\;\;\;\;\xi>0, \;x\in C
    \end{array}
\end{equation*}
\end{center}
\vspace{-0.1cm}
In turn, by Lemma \ref{lem:3.2}, \hyperlink{p2}{$(P2)$} is equivalent to the problem
\begin{center} \hypertarget{p3}{}
\vspace{-0.7cm}\begin{equation*}{(P3)\;\;\;\;\;\;\;\;}
\begin{array}{ll}
    \inf\limits_{x'}\;\;\langle \alpha,x'\rangle\\
    \hspace{0.05cm}\text{s.t.}\; \langle y,Ax'\rangle\geq 1\;\;\forall y\in T\\ \;\;\;\;\;\;\;x'\in C\setminus\mathcal{D}
    \end{array}
\end{equation*}
\end{center}
where $\mathcal{D}:=  \{x'\in C:\langle \alpha,x'\rangle\leq 0\}$. Moreover, as $\text{val}(P2)=\underline{v}<+\infty$, we have $\text{val}(P3)=\frac{1}{\text{val}(P2)}$ if the former (or latter) is feasible. Observe that for every $y\in T$ we have $\langle y,\beta\rangle\geq 1$, and the inclusion $K^*\subset \{z\in Z:\langle y,z\rangle \geq 0\;\forall y\in T\}$ holds. Therefore, if we consider the restricted conic linear program
\begin{center}\hypertarget{p4}{}
\vspace{-0.7cm}\begin{equation*}{(P4)\;\;\;\;\;\;\;\;}
\begin{array}{ll}
    \inf\;\; \langle \alpha,x'\rangle\\
    \hspace{0.05cm}\text{s.t.}\;\; Ax'-\beta\in K^*\\ \;\;\;\;\;\;\;x'\in C\setminus\mathcal{D}
    \end{array}\;\;\;\;\;\;\;\;\;
\end{equation*}
\end{center}
then the feasibility of \hyperlink{p4}{$(P4)$} implies that of \hyperlink{p3}{$(P3)$}, and the inequality $\text{val}(P4)\geq \text{val}(P3)$ holds.\par
Consider now the program \hyperlink{d1}{$(D1)$}. By Lemmas \ref{lem:securitylevels}, \ref{lem:basic} we find $\overline{v}=\text{val}(D1)=\text{val}(D2)$, where \hypertarget{d2}{}
\begin{center}
\vspace{-0.7cm}\begin{equation*}{\;(D2)\;\;\;\;\;\;\;\;}
\begin{array}{ll}
    \inf\limits_{\zeta,y}\;\; \zeta\\
    \hspace{0.05cm}\text{s.t.}\;\;\langle A^*y,x\rangle\leq\zeta\;\;\forall x\in S\\
    \;\;\;\;\;\;\;\hspace{0.04cm}\langle y,\beta\rangle=1\\
    \;\;\;\;\;\;\;\;\zeta>0,\;y\in K
    \end{array}\;
\end{equation*}
\end{center} 

 Verbatim the equivalence shown between problems \hyperlink{p2}{$(P2)$} and \hyperlink{p3}{$(P3)$}, one can prove a dual counterpart of Lemma \ref{lem:3.2} to show that \hyperlink{d2}{$(D2)$} is equivalent to the program\hypertarget{d3}{}
\begin{center}
\vspace{-0.7cm}\begin{equation*}{\;(D3)\;\;\;\;\;\;\;\;\;}
\begin{array}{ll}
    \sup\limits_{y'}\;\;\langle y',\beta\rangle\\
    \hspace{0.05cm}\text{s.t.}\;\;\langle A^*y',x\rangle\leq 1\;\;\forall x\in S\\ \;\;\;\;\;\;\;\;y'\in K\setminus\mathcal{E}
    \end{array}
\end{equation*}
\end{center}
where $\mathcal{E}:=  \{y'\in K: \langle y',\beta\rangle\leq0\}$ (the statement and proof of such counterpart are omitted, as the latter follows directly from the one given for the lemma by reversing the inequality and switching ``$\sup$" and ``$\inf$").  Next, consider the program 

\begin{center}\hypertarget{d4}{}
\vspace{-0.7cm}\begin{equation*}{\;(D4)\;\;\;\;\;\;\;\;\;}
\begin{array}{ll}
    \sup\;\; \langle y',\beta\rangle\\
    \hspace{0.05cm}\text{s.t.}\;\; -A^*y'+\alpha\in C^*\\ \;\;\;\;\;\;\;\;y'\in K\setminus\mathcal{E}
    \end{array}\;\;\;\;\;
    \end{equation*}
\end{center}
The inclusion $S\subset C$ implies the inclusion $C^*\subset \{w\in W:\langle w,x\rangle\geq0\;\forall x\in S\}$. In addition, for every $x\in S$ the inequality $\langle\alpha,x\rangle\leq1$ holds. As a consequence, we have $\text{val}(D4)\leq \text{val}(D3)$, and the feasibility of \hyperlink{d4}{$(D4)$} implies that of \hyperlink{d3}{$(D3)$}.\par

 By assumption, a zero duality gap exists for the pair \hyperlink{newprimaldual}{$\{(P_\beta),(D_\alpha)\}$}, i.e., $\text{val}(P_\beta)=\text{val}(D_\alpha)>0$, and both are feasible. Thus, both \hyperlink{p4}{$(P4)$} and \hyperlink{d4}{$(D4)$} are feasible. In fact, $\text{val}(D_\alpha)=\text{val}(D4)$ and $\text{val}(P_\beta)=\text{val}(P4)$. Therefore, \hyperlink{p3}{$(P3)$} is also feasible, and by Lemma \ref{lem:3.2} the relation $\text{val}(P3)=\frac{1}{\text{val}(P2)}$ holds. The following inequality is immediately obtained:
\begin{equation}\label{eq:3.1}
    \text{val}(P_\beta)\geq\displaystyle\frac{1}{\underline{v}}\;.
\end{equation} 
Similarly, we get $\text{val}(D3)=\frac{1}{\text{val}(D2)}$, which yields the bound
\begin{equation}\label{eq:3.2}
    \text{val}(D_\alpha)\leq \displaystyle\frac{1}{\overline{v}}\; .
\end{equation}
It follows from (\ref{eq:3.1}) and (\ref{eq:3.2}) that $\underline{v}\geq\overline{v}$. By (\ref{eq:2.5}) we finally acquire $\underline{v}=\overline{v}=\frac{1}{\text{val}(P_\beta)}$.\par

 The existence of optimal strategies for players $I$ and $II$ eventually boils down to the existence of optimal solutions for \hyperlink{primalbeta}{$(P_\beta)$} and \hyperlink{dualalpha}{$(D_\alpha)$}, respectively. Let $x^*\in C$ be an optimal solution for \hyperlink{primalbeta}{$(P_\beta)$}. Then, $x^*$ is also an optimal solution for \hyperlink{p3}{$(P3)$} and
\begin{equation}\label{eq:3.3}
   v=\displaystyle\frac{1}{\text{val}(P3)}=\frac{1}{\langle \alpha,x^*\rangle}. 
\end{equation}
Define $\widetilde{x}:=  vx^* $. Then $\widetilde{x}$ is an optimal solution for \hyperlink{p2}{$(P2)$} by (\ref{eq:3.3}) and Lemma \ref{lem:3.2}. As $\widetilde{x}\in S_1\subset S$ and $\text{val}(P1)=\text{val}(P2)$, by Lemma \ref{lem:securitylevels} we deduce that $\widetilde{x}\in\mathcal{B}^{I}$, that is, $\widetilde{x}=vx^*$ is an optimal strategy for player $I$. Similarly, if $y^*$ is an optimal solution of \hyperlink{dualalpha}{$(D_\alpha)$}, then $\widetilde{y}:=  vy^*$ is an optimal strategy for player $II$. Consequently, $(\widetilde{x},\widetilde{y})$ is a saddle point.
\end{proof}

\vspace{-0.1cm}\subsection{Games with strategy sets that are bases of convex cones}\label{sec:3.2}\vspace{-0.2cm}

We now show that the minimax theorem is a direct consequence of strong duality in the case where the strategy sets of the players represent bases of convex cones. This is a special case of Theorem \ref{th:3.1}, but the positive-sign, feasibility and existence of a zero duality gap assumptions are no longer needed.\par
 From now on we assume that the spaces $X,\;Y$ are reflexive Banach spaces, paired with their dual (continuous) reflexive Banach spaces $X^*,\;Y^*$, respectively. We equip all spaces with their corresponding (w.r.t. the pair) strong topologies. Moreover, we assume that the convex cones $C,\;K$ are closed and that the cones $C^*,\;K^*$ are solid. In addition, for a game $G=(S,T,u)$, when the strategy sets $S,\;T$ are bases of the cones $C,\;K$ generated by $\alpha\in \text{int}C^*$, $\beta\in \text{int}K^*$ (and $H=Q=\{1\}$), respectively, we will denote it by $G=(\alpha,\beta,A)$, or simply $G_A$ when the normal vectors $\alpha,\;\beta$ are considered known and fixed; or when we want to make the difference between games with identical strategy sets but different payoff operators clear.

\vspace{0.3cm}\begin{theorem}[Minimax theorem for bases]\label{th:3.3}
    Consider a two-player zero-sum game $G=(\alpha,\beta,A)$ defined by \hyperlink{f1}{(F1)}-\hyperlink{f3}{(F3)} for functionals $\alpha\in \text{int}C^*,\;\beta\in\text{int}K^*$.    Then, $\underline{v}=\overline{v}$ and the set $\mathcal{B}^{I}\times\mathcal{B}^{II}$ of Nash equilibria is non-empty, convex and bounded.
\end{theorem}

 For its proof, we will need some additional Lemmas. The first one is immediately derived from a characterization of reflexivity of Banach spaces and the fact that $\text{int} C^*=\text{int} C^{*s}$:
\begin{lemma}[\normalfont{Theorem 3.6  \cite{casini2010cones}}]\label{lem:3.4}
    Let $X$ be a reflexive Banach space and $C$ be a closed convex cone in $X$. If $\text{int} C^*\neq\emptyset$, then $\text{int} C^*=C^{*s}$.
\end{lemma}

\begin{lemma}\label{lem:3.5}
    Consider two closed convex cones $C\subseteq X$, $K\subseteq Y$, a linear operator $A:X\to Y^*$, a functional $\beta \in \text{int}K^*$, and some $r\in\mathbb{R}$. Further, consider the base $T:= \{y\in K:\langle y,\beta\rangle=1\}$. Then, for any $x\in C$ the following equivalence holds:
    \begin{equation*}
        \langle y,Ax\rangle\geq r\;\;\forall y\in T \;\;\Longleftrightarrow\;\;\langle y,Ax\rangle\geq r\langle y,\beta\rangle\;\;\forall y\in K.
    \end{equation*}
\end{lemma}
\begin{proof}
    Since $\beta$ belongs to $\text{int} K^*$, and $Y^*$ is a reflexive Banach space, by Lemma \ref{lem:3.4} we have that $\langle y,\beta\rangle >0$ for every $y\in K\setminus\{0\}$. This implies that $T=\left\{\frac{y}{\langle y,\beta\rangle}: y\in K\setminus\{0\}\right\}$. The result directly follows.
\end{proof}

Next, we show that the set of Nash equilibria is invariant under positive translations of the linear payoff operator for every game with strategy sets that are bases of convex cones.

\begin{lemma}\label{lem:3.6}
     Let $\alpha\in \text{int}C^*$, $\beta\in \text{int}K^*$, and a linear operator $A:X\rightarrow Y^*$. Consider the linear operators $E:X\rightarrow Y^*$ with $E(x):= \beta\langle\alpha,x\rangle$, and $B:X\rightarrow Y^*$ with $B(x):=\lambda A(x)+\kappa E(x)$ for some fixed $\lambda,\kappa >0$. Also, let $v_A$ be the game value for $G_A=(\alpha,\beta,A)$ and $v_B$ be the corresponding one for $G_B=(\alpha,\beta,B)$, where both games are defined by \hyperlink{f1}{(F1)}-\hyperlink{f3}{(F3)}. Then, $v_A$ exists if and only if $v_B$ exists, and in that case $v_B=\lambda v_A+\kappa$. Moreover, $(x^*,y^*)$ is a Nash equilibrium for $G_A$ if and only if it is a Nash equilibrium for $G_B$.
\end{lemma}
\begin{proof}
As the strategy sets are bases of convex cones, the following relation holds:
\begin{equation}\label{eq:B}
    \langle y,Bx\rangle=\lambda\langle y,Ax\rangle +\kappa\langle y,\beta\rangle\langle \alpha,x\rangle=\lambda\langle y,Ax\rangle+\kappa\;\;\forall x\in S,\;\;\forall y\in T.
\end{equation}
Let $\underline{v}_A$ be the lower value for $G_A$ and $\underline{v}_B$ be the corresponding one for $G_B$. Then, (\ref{eq:B}) yields
\begin{equation*}
    \underline{v}_B=\sup_{x\in S}\inf_{y\in T}\langle y,Bx\rangle=\lambda\sup_{x\in S}\inf_{y\in T}\langle y,Ax\rangle +\kappa =\lambda\underline{v}_A+\kappa.
\end{equation*}
Similarly, if $\overline{v}_A,\;\overline{v}_B$ are the upper values for $G_A$ and $G_B$, respectively, then $\overline{v}_B=\lambda\overline{v}_A+\kappa$. Therefore, $v_B$ exists if and only if $v_A$ exists, and the relation $v_B=\lambda v_A+\kappa$ follows.\par
For the second part, let $(x^*,y^*)$ be a Nash equilibrium for $G_A$. Then, by (\ref{eq:2.6}) and (\ref{eq:B}) we equivalently have
\begin{equation*}
    \langle y,Ax^*\rangle\geq\langle y^*,Ax^*\rangle\geq\langle y^*,Ax\rangle\;\;\forall x\in S,\;\;\forall y\in T\;\Leftrightarrow
\end{equation*}
\begin{equation*}
    \lambda\langle y,Ax^*\rangle+\kappa\geq\lambda\langle y^*,Ax^*\rangle+\kappa\geq\lambda\langle y^*,Ax\rangle+\kappa\;\;\forall x\in S,\;\;\forall y\in T\;\Leftrightarrow
\end{equation*}
\begin{equation*}
    \langle y,Bx^*\rangle\geq\langle y^*,Bx^*\rangle\geq\langle y^*,Bx\rangle\;\;\forall x\in S,\;\;\forall y\in T,
\end{equation*}
which is equivalent to $(x^*,y^*)$ being a Nash equilibrium for $G_B$.
\end{proof}

\begin{proof}[\textit{Proof of Theorem 2}] Let $\lambda>0$ such that $-1/2\leq\lambda\underline{v}\leq\lambda\overline{v}\leq 1/2$. Consider the linear operator $B_\kappa:=  \lambda A +\kappa E$, where $E$ is as in Lemma \ref{lem:3.6}, and the game $G_{B_\kappa}=(\alpha,\beta,B_\kappa)$, where $\kappa$ ranges in $(1/2,+\infty)$. Its adjoint operator is given by $B^*_\kappa(y)=\lambda A^*y+\kappa\alpha\langle y,\beta\rangle$ for $y\in Y$, as it satisfies the relation $\langle y,B_\kappa (x)\rangle=\langle B^*_\kappa (y),x\rangle$ for all $x\in X$. Moreover, if $\underline{v_\kappa}$ is the lower value of the game $G_{B_\kappa}$, then $\underline{v_\kappa}>0$ for every $\kappa\in(1/2,+\infty)$ due to Lemma \ref{lem:3.6}.\par

In turn, consider the problems \hyperlink{p1}{$(P1)$}--\hyperlink{p4}{$(P4)$} and \hyperlink{d1}{$(D1)$}--\hyperlink{d4}{$(D4)$} as in Lemma \ref{lem:securitylevels} and Theorem \ref{th:3.1}, but replace $A$ with the payoff operator $B_\kappa$. We denote these programs by $(P^\kappa 1)$--$(P^\kappa 4)$ and $(D^\kappa 1)$--$(D^\kappa 4)$, to make the distinction clear. The program $(P^\kappa 1)$ is feasible and it is exactly the same program as $(P^\kappa 2)$ by definition of $S$. Hence, $(P^\kappa 1)$ is equivalent to $(P^\kappa 3)$ by Lemma \ref{lem:3.2}. In turn, by Lemma \ref{lem:3.5}, $(P^\kappa 3)$ and $(P^\kappa 4)$ are the same conic programs. Thus, $\text{val}(P^\kappa 4)=\text{val}(P^\kappa 1)=1/{\underline{v_\kappa}}$, with the last equality resulting from Lemma \ref{lem:securitylevels}. Similarly, $(D^\kappa 1)$ is feasible (by \hyperlink{f3}{(F3)}), it is equivalent to $(D^\kappa 4)$, and $1/\overline{v_\kappa}=\text{val}(D^\kappa 1)=\text{val}(D^\kappa 4)$ by Lemma \ref{lem:securitylevels}.\par
 Consider the primal-dual pair $\{(P_\beta^\kappa),(D_\alpha^\kappa)\}$, which is identical to the pair \hyperlink{newprimaldual}{$\{(P_\beta),(D_\alpha)\}$}, but $A$ is replaced by $B_\kappa$. Then $\text{val}(P_\beta^\kappa)\geq \text{val}(D_\alpha^\kappa)\geq \text{val}(D^\kappa 4)>0$. Moreover, both $(P_\beta^\kappa)$ and $(D_\alpha^\kappa)$ are feasible, and the equalities $\text{val}(P_\beta^\kappa)=\text{val}(P^\kappa 4)$ and $\text{val}(D_\alpha^\kappa)=\text{val}(D^\kappa 4)$ hold. Indeed, the former is straightforward by the fact that $\text{val}(P_\beta^\kappa)>0$ ($\mathcal{D}=\{0\}$ and $x=0$ is not feasible due to Lemma \ref{lem:3.4}), while the latter occurs from the fact that $(D^\kappa 4)$ is feasible and $\text{val}(D^\kappa_\alpha)>0$. Consequently, $\text{val}(P^{\kappa}_\beta)=1/\underline{v_\kappa}$ and $\text{val}(D^{\kappa}_{\alpha})=1/\overline{v_\kappa}$. Now, note that $(P_\beta^\kappa)$ and $(D_\alpha^\kappa)$ take the following forms:
\begin{center}\hypertarget{primalbetakappa}{}
\vspace{-0.7cm}\begin{equation*}{(P_\beta^\kappa)\;\;\;\;\;\;\;\;}
\begin{array}{ll}
    \inf\limits_{x}\;\; \langle \alpha,x\rangle\\
    \hspace{0.05cm}\text{s.t.}\;\; \lambda Ax+\beta(\kappa\langle\alpha,x\rangle-1)\in K^*\\ \;\;\;\;\;\;\;x\in C
    \end{array}
\end{equation*}
\end{center}
\begin{center}\hypertarget{dualalphakappa}{}
\vspace{-0.7cm}\begin{equation*}{\;\;\;\;\;\;(D_\alpha^\kappa)\;\;\;\;\;\;\;\;}
\begin{array}{ll}
    \sup\limits_{y}\;\; \langle y,\beta\rangle\\
    \hspace{0.05cm}\text{s.t.}\;\; -\lambda A^*y+\alpha(1-\kappa\langle y,\beta\rangle)\in C^*\\ \;\;\;\;\;\;\;\;\hspace{0.03cm}y\in K
    \end{array}
    \end{equation*}
\end{center}
Fix $\kappa\in(1/2,1)$. Then, for $y\in T$, one can find suitable $\varepsilon>0$ such that $\varepsilon y\in K\setminus\{0\}$ is a strictly feasible solution for $(D_\alpha^\kappa)$. Indeed, since $\alpha\in\text{int}C^*$, for any $y\in T$ there exists $\delta=\delta(A^*, y,\lambda,\kappa)\in(0,1)$ sufficiently small such that $-\delta A^*\left(\frac{\lambda y}{1-\kappa}\right)+\alpha\in\text{int}C^*$. Then, for $\varepsilon:= \frac{\delta}{1-\kappa+\kappa\delta}$ we have $-\lambda A^*(\varepsilon y)+\alpha(1-\varepsilon\kappa)\in \text{int}C^*$. From the dual counterpart of Proposition \ref{prop:2.1} (see Remark \ref{rem:reflexive}) we deduce that $\text{val}(D_\alpha^\kappa)=\text{val}(P_\beta^\kappa)$ and that the set $\mathcal{S}(P_\beta^\kappa)$ is non-empty, bounded and convex. Hence, $\underline{v_\kappa}=\overline{v^\kappa}=\frac{1}{\text{val}(P^\kappa_\beta)}=v_{\kappa}$. Lemma \ref{lem:3.6} finally gives $\underline{v}=\overline{v}$.\par

It remains to show that the set of Nash equilibria is non-empty, convex and bounded. To this end, observe that $x^*$ is an optimal solution for $(P^\kappa_\beta)$ if and only if it is an optimal solution for $(P^\kappa 3)$. In turn, $x^*$ is feasible for $(P^\kappa 3)$ if and only if $x^*v_\kappa$ is feasible for $(P^\kappa 2)$, and the optimal value of each program is attained for that feasible point (Lemma \ref{lem:3.2}). It is easy to see that $(P^\kappa 2)$ and $(P^\kappa 1)$ share the same set of optimal solutions. Therefore, since every optimal solution of $(P^\kappa 1)$ is an optimal strategy for player $I$ in $G_{B_\kappa}$ and vice versa (due to Lemma \ref{lem:securitylevels}), we get $\mathcal{B}^I=v_\kappa \mathcal{S}(P^\kappa_\beta)$ from Lemma \ref{lem:3.6}. As we've already shown that the set $\mathcal{S}(P^\kappa_\beta)$ has the desired properties, $\mathcal{B}^{I}$ is also non-empty, bounded and convex due to $v_\kappa>0$.\par

To show that $\mathcal{B}^{II}$ has the same properties, as previously, one can fix $\kappa'$ large enough such that $(P_\beta^{\kappa'})$ is strictly feasible. Then, again by Proposition \ref{prop:2.1}, Lemmas \ref{lem:securitylevels}, \ref{lem:3.2} and \ref{lem:3.6}, the same analysis -- now for player $II$ of the game $G_{B_{\kappa'}}$ -- gives $\underline{v}=\overline{v}$ and $\mathcal{B}^{II}=v_{\kappa'}\mathcal{S}(D^{\kappa'}_\alpha)$, where $v_{\kappa'}$ is the game value of $G_{B_{\kappa'}}$ and $\mathcal{S}(D^{\kappa'}_\alpha)$ is non-empty, convex and bounded. As $v_{\kappa'}>0$, the set of optimal strategies $\mathcal{B}^{II}$ is also non-empty, convex and bounded, and the result follows.
\end{proof}
\begin{remark}
    As seen by the proof, for a suitable ``translation" of $G_A$ (choice of $\kappa,\kappa'$) one can find an explicit representation for both the game value and the set of Nash equilibria of the initial game, which depends merely on the solutions of certain conic programs.
\end{remark}

\vspace{-0.1cm}\subsection{Extensions to games with homogeneous payoff functions}\label{sec:3.3}\vspace{-0.2cm}

  With some slight modifications in its proof, the first main result, Theorem \ref{th:3.1}, can be extended to zero-sum games with positive-homogeneous (with respect to each variable) payoff functions, i.e., to payoff functions $u:X\times Y\to\mathbb{R}$ with the property $u(\lambda x, y)=\lambda^k u(x,y)$ and $u(x,\lambda y)=\lambda^l u(x,y)$ for some integers $k,l$ and every $\lambda>0$. To be more precise, note that the positive homogeneity of $u(x,y)=\langle y,Ax\rangle$ was used to show equivalence between programs \hyperlink{p2}{$(P2)$} and \hyperlink{p3}{$(P3)$} in Lemma \ref{lem:3.2}, while linearity was used to restrict problem \hyperlink{p3}{$(P3)$} into the necessary conic form that is characterized by the appearance of the dual convex cone $K^*$, namely program \hyperlink{primalbeta}{$(P_\beta)$}. In fact, bilinearity of the payoff function is not needed, and homogeneity is the only sufficient property for any two-player zero-sum game to be described by a pair of conic convex programs of the form \hyperlink{p2}{$(P2)$}--\hyperlink{p3}{$(P3)$} and \hyperlink{d2}{$(D2)$}--\hyperlink{d3}{$(D3)$}. The cost of such minimal requirement is that we ``lose" the duality of the convex cones $C$ and $K$, that is, the cones $C^*$ and $K^*$ no longer appear, thus many certificates (e.g., Proposition \ref{prop:2.1}) for the existence of a zero duality gap and optimal solutions cannot be invoked. However, given that such a gap between \hyperlink{p1}{$(P1)$} and \hyperlink{d1}{$(D1)$} is absent (for instance, via invoking Lemma \ref{lem:securitylevels} and known fixed-point theorems \cite{fan1952fixed, glicksberg1952further}), one can then apply suitable algorithms (e.g., interior-point, trust-region, Newton methods) to the corresponding positive-homogeneous conic versions of \hyperlink{p2}{$(P2)$}--\hyperlink{p3}{$(P3)$} and \hyperlink{d2}{$(D2)$}--\hyperlink{d3}{$(D3)$}, which are oftentimes much more efficient in practice than fixed-point and saddle-point ones in many conic optimization settings, such as in quadratic and semidefinite programming \cite{nesterov1994interior, nocedal1999numerical, wolkowicz2012handbook}.

\vspace{-0.1cm}\section{The Minimax Theorem proves Strong Duality almost always}\label{sec:4}\vspace{-0.2cm}

We now show that the converse is ``almost always" true when the strategy spaces are bases of convex cones, namely, we prove that minimax is \textit{almost equivalent} to strong duality for the pair \hyperlink{primaldual}{$\{(P),(D)\}$}. In what follows we still assume that the spaces $X,\;Y$ are reflexive Banach spaces, paired with their dual reflexive Banach spaces $X^*,\;Y^*$, respectively, and that the cones $C^*,\;K^*$ are solid. Further, for every $c\in X^*$, $b\in Y^*$, we define the null sets (orthogonal complements)
\begin{subequations}
\begin{align}
c^{\perp}&:= \{x\in X:\langle c,x\rangle=0\},\\ 
b^{\perp}&:= \{y\in Y: \langle y,b\rangle=0\}.
\end{align}
\end{subequations}

\begin{theorem}\label{th:4.1}
    Consider a linear operator $A:X\rightarrow Y^*$, closed convex cones $C\subseteq X,\;K\subseteq Y$ and functionals $c\in X^*,\;b\in Y^*$ such that the primal-dual pair \hyperlink{primaldual}{$\{(P),(D)\}$} is feasible. Also, consider the game $G=(\alpha,\beta,A)$ for some $\alpha\in\text{int} C^*,\;\beta\in\text{int}K^*$, and let $v$ be its game value. The following holds:\vspace{-0.2cm}
    \begin{enumerate}[(i)]
        \item If $v\neq0$, then at least one of \hyperlink{primal}{$(P)$}, \hyperlink{dual}{$(D)$} is strictly feasible. In fact, $\text{val}(P)=\text{val}(D)$ and at least one of the sets $\mathcal{S}(P),\;\mathcal{S}(D)$ is non-empty, convex, and bounded.\vspace{-0.1cm}
        \item If $v=0$ and either $\mathcal{B}^{I}\cap c^{\perp}=\emptyset$ or $\mathcal{B}^{II}\cap b^{\perp}=\emptyset$, then at least one of \hyperlink{primal}{$(P)$}, \hyperlink{dual}{$(D)$} is strictly feasible. In fact, $\text{val}(P)=\text{val}(D)$ and at least one of the sets $\mathcal{S}(P),\;\mathcal{S}(D)$ is non-empty, convex, and bounded.\vspace{-0.1cm}
        \item If $v=0$, $\mathcal{B}^{I}\cap c^{\perp}\neq\emptyset$ and $\mathcal{B}^{II}\cap b^{\perp}\neq\emptyset$, then neither of the problems \hyperlink{primal}{$(P)$}, \hyperlink{dual}{$(D)$} is strictly feasible. In addition, if $\mathcal{B}^{I}\cap c^{\perp}\cap\mathcal{F}(P)\neq\emptyset$ and $\mathcal{B}^{II}\cap  b^{\perp}\cap\mathcal{F}(D)\neq\emptyset$, then $\text{val}(P)=\text{val}(D)=0$ and $\mathcal{B}^{I}\cap\mathcal{S}(P)\neq\emptyset$, $\mathcal{B}^{II}\cap\mathcal{S}(D)\neq\emptyset$.\\
    \end{enumerate}
\end{theorem}

\begin{remark}\label{rem:4} (i) The only case not covered, and thus the one that justifies the term ``almost" in the \textit{almost equivalence}, is the one where $v=0$, $\mathcal{B}^{I}\cap c^{\perp}\neq\emptyset$, $\mathcal{B}^{II}\cap b^{\perp}\neq\emptyset$, and at least one of the sets $\mathcal{B}^{I}\cap c^{\perp}\cap\mathcal{F}(P)$, $\mathcal{B}^{II}\cap b^{\perp}\cap\mathcal{F}(D)$ is empty for any chosen gane, i.e., for any initial choice of $\alpha\in\text{int}C^*$ and $\beta\in\text{int}K^*$. In this case, the absence of a zero-duality gap does not depend on which of the latter two sets is empty (Example \ref{ex:4.4}). Further, depending on the sign of the game value $v$, one can determine exactly which of the problems \hyperlink{primal}{$(P)$} and \hyperlink{dual}{$(D)$} is strictly feasible in part $(i)$ of Theorem \ref{th:4.1}. In similar fashion, one can also determine exactly which one is strictly feasible in part $(ii)$, by checking the infeasibility of the sets $\mathcal{B}^{I}\cap c^{\perp}$, $\mathcal{B}^{II}\cap b^{\perp}$.\par
(ii) A significant observation can be made regarding strong duality for the pair \hyperlink{primaldual}{$\{(P),(D)\}$}: If the chosen (w.r.t. $\alpha,\;\beta$) game has a non-zero value, then strong duality holds for any such feasible primal-dual pair, no matter what form $b$ and $c$ take. In particular, it depends only on the payoff operator $A$, as it determines the value of the game. On the other hand, if the game value is zero, then strong duality (strict feasibility and the complementarity conditions to be exact) for the pair \hyperlink{primaldual}{$\{(P),(D)\}$} depends solely on $b$, $c$, the feasibility sets and the optimal strategies of the players.
\end{remark}

As already hinted at the introduction and parts $(i)-(iii)$ above, in order to prove the almost equivalence we consider two cases regarding the value $v$ of the chosen game $G=(\alpha,\beta,A)$. First, a theorem of the alternative for the case $v\neq 0$ is stated and proved.

\begin{lemma}\label{lem:4.2}
    Consider a linear operator $A:X\rightarrow Y^*$ and closed convex cones $C\subseteq X,\;K\subseteq Y$. Also, consider the game $G=(\alpha,\beta,A)$ for some $\alpha\in \text{int} C^*,\;\beta\in\text{int}K^*$, and let $v$ be its game value. Then, $v\neq0$ if and only if exactly one of the following is true:\vspace{-0.3cm}
    \begin{enumerate}[(i)]
        \item $\exists y\in K\setminus\{0\}$ such that $-A^*y\in \text{int}C^*$,\vspace{-0.2cm}
        \item $\exists x\in C\setminus\{0\}$ such that $Ax\in \text{int}K^*$.\quad
    \end{enumerate}
\end{lemma}

\begin{proof}
We first show that $(i)$ and $(ii)$ cannot both be true at the same time. Suppose that $\exists y\in K\setminus\{0\}$ and $\exists x\in C\setminus\{0\}$ such that $-A^*y\in \text{int}C^*$ and $Ax\in \text{int}K^*$. Then, by Lemma \ref{lem:3.4} we have that
\begin{equation}\label{eq:4.2}
        0< \langle -A^*y,x\rangle = - \langle y,Ax\rangle < 0,
\end{equation}
which is a contradiction. For the ``only if" direction, let $v\neq0$ and suppose that there doesn't exist any $y\in K\setminus\{0\}$ such that $-A^*y\in \text{int}C^*$ (else there is nothing to show). We claim that $v \geq0$. Indeed, assume otherwise and take $y^*\in\mathcal{B}^{II}$ (such an optimal solution exists due to Theorem \ref{th:3.3}). Then, $\langle y^*,Ax\rangle\leq v \langle \alpha,x\rangle$ holds for every $x\in C\setminus\{0\}$, due to the dual counterpart of Lemma \ref{lem:3.5}. Therefore, $\langle A^*y^*,x\rangle< 0$ for every $x\in C\setminus\{0\}$ by Lemma \ref{lem:3.4}. This implies that $-A^*y^*\in \text{int}C^*$, which is a contradiction. Hence $v\geq0$. But $v\neq0$, so $v>0$. For any $x^*\in\mathcal{B}^{I}$ we have $\langle y,Ax^*\rangle \geq v\langle y,\beta\rangle>0$ $\forall y\in K\setminus\{0\}$ by Lemma \ref{lem:3.5}. It is straightforward, again by Lemma \ref{lem:3.4}, that $Ax^*\in \text{int}K^*$.\par
     For the opposite direction, suppose that $v=0$. Then, by (\ref{eq:2.6}) and Theorem \ref{th:3.3}, there exist $x^*\in S$, $y^*\in T$ such that 
    \begin{equation}\label{eq:4.3}
        \langle A^*y,x^*\rangle \geq 0\geq \langle A^*y^*,x\rangle,\;\;\forall x\in S,\;\forall y\in T.
    \end{equation}
    Assume that $\exists y\in K\setminus\{0\}$ such that $-A^*y\in \text{int}C^*$. Then, $\hat{y}:= \displaystyle\frac{y}{\langle y,\beta\rangle}\in T$ and $-A^*\hat{y}\in \text{int}C^*$ by Lemma \ref{lem:3.4}. Hence $\langle A^*\hat{y},x^*\rangle<0$, which contradicts (\ref{eq:4.3}). Similarly, if $(ii)$ holds, we can find $\hat{x}\in S$ so that (\ref{eq:4.3}) is violated.    
\end{proof}

It should be mentioned that, with a slight modification in the proof, the sets $C\setminus\{0\}$, $K\setminus\{0\}$ of parts $(i)$, $(ii)$ can be replaced by $\mathcal{B}^I$, $\mathcal{B}^{II}$, respectively, since the latter are non-empty by Theorem \ref{th:3.3}. Next, we give a characterization of strict feasibility when the value of the chosen game is zero.

\vspace{0.1cm}\begin{lemma}\label{lem:4.3}
    Consider a linear operator $A:X\rightarrow Y^*$ and closed convex cones $C\subseteq X,\;K\subseteq Y$, such that the primal-dual pair \hyperlink{primaldual}{$\{(P),(D)\}$} is feasible. Also, consider the game $G=(\alpha,\beta,A)$ for some $\alpha\in\text{int} C^*,\;\beta\in\text{int}K^*$, and let $v$ be its game value. If $v=0$, then the following equivalences hold:
    \begin{equation}\label{eq:4.4}
        (P) \mbox{\;\;is strictly feasible\;\;} \Longleftrightarrow\; \forall y\in T: -A^*y\in C^*\Rightarrow \langle y,b\rangle<0.
    \end{equation}
    \begin{equation}\label{eq:4.5}
        (D) \mbox{\;\;is strictly feasible\;\;} \Longleftrightarrow\; \forall x\in S: Ax\in K^*\Rightarrow \langle c,x\rangle>0.\;\;\;\;
    \end{equation}
\end{lemma}

\begin{proof}
    We will only show (\ref{eq:4.4}) (the dual relation (\ref{eq:4.5}) is similarly proved, and the dual counterparts of Lemma \ref{lem:3.4} and Proposition \ref{prop:2.1} are used; see Remark \ref{rem:reflexive}). Before proving the two directions of (\ref{eq:4.4}) separately, we begin with an important observation: Given $v=0$, Theorem \ref{th:3.3} implies that there exist $x^*\in S$ and $y^*\in T$ such that (\ref{eq:4.3}) holds. Hence, the intersections $S\cap\{x\in C:Ax\in K^*\}$ and $T\cap\{y\in K: -A^*y\in C^*\}$ are non-empty. In addition, for every $x\in S$ such that $Ax\in K^*$ we have $\langle c,x\rangle \geq0$. To see this, suppose that there exists some $x\in S$ such that $Ax\in K^*$ and $\langle c,x\rangle<0$. If $\tilde{x}\in C$ is feasible for \hyperlink{primal}{$(P)$}, then $\lambda x+\tilde{x}$ is also a feasible solution for every $\lambda>0$. By taking $\lambda\rightarrow +\infty$ we obtain $\text{val}(P)=-\infty$, which is impossible due to the feasibility of \hyperlink{dual}{$(D)$}. Similarly, for every $y\in T$ such that $-A^*y\in C^*$ we have $\langle y,b\rangle\leq0$.\par
 Now, let $y^*\in T$
 such that $-A^*y^*\in C^*$ and assume that \hyperlink{primal}{$(P)$} is strictly feasible. Then, there exists $x^*\in C$ such that $Ax^*-b\in \text{int}K^*$. By Lemma \ref{lem:3.4} we find
 \begin{equation}
     0\geq \langle A^*y^*,x^*\rangle=\langle y^*,Ax^*\rangle>\langle y^*,b\rangle.
 \end{equation}
 For the other direction, suppose that \hyperlink{primal}{$(P)$} is not strictly feasible. We define the linear operator $A_\beta:X\times\mathbb{R}\rightarrow Y^*$ by $A_\beta(x,\kappa):= Ax+\kappa\beta$.
   Consider the primal program
   \begin{center}\hypertarget{primaltilde}{}
\vspace{-0.7cm}\begin{equation*}{(\widetilde{P})\;\;\;\;\;\;\;\;}
\begin{array}{ll}
    \inf\limits_{x,\kappa}\;\; \langle \widetilde{c},(x,\kappa)\rangle\\
    \text{s.t.}\;\; A_\beta(x,\kappa)-b\in K^*\\ \;\;\;\;\;\;(x,\kappa)\in C\times\mathbb{R}
    \end{array}
\end{equation*}
\end{center}
For $\widetilde{c}:= (0_W,1)$, the objective function is equal to $\kappa$. Since \hyperlink{primal}{$(P)$} is feasible, \hyperlink{primaltilde}{$(\widetilde{P})$} is also feasible. In addition, it is strictly feasible: For $x=0_X$, one can find $\kappa>0$ large enough such that $\kappa\beta-b\in \text{int}K^*$. Also, its objective value is always non-negative: Suppose that there exist $x\in C$, $\kappa<0$ such that $Ax+\kappa\beta-b\in K^*$. Then $Ax+\kappa\beta -b -\kappa\beta\in \text{int}K^*$, which is a contradiction, since \hyperlink{primal}{$(P)$} is not strictly feasible. Thus, $\text{val}(\widetilde{P})$ is attained for any $x\in \mathcal{F}(P)$ and $\kappa=0$, and it is equal to $0$. Now, the dual of \hyperlink{primaltilde}{$(\widetilde{P})$} is the following:
\begin{center}\hypertarget{dualtilde}{}
\vspace{-0.7cm}\begin{equation*}{(\widetilde{D})\;\;\;\;\;\;\;\;}
\begin{array}{ll}
    \sup\limits_{y}\;\; \langle y,b\rangle\\
    \text{s.t.}\;\; -A^*y\in C^*
    \\ \;\;\;\;\;\;\;\hspace{0.06cm}-\langle y,\beta\rangle+1=0\\ \;\;\;\;\;\;\;\;y\in K
    \end{array}
    \end{equation*}
\end{center}

This is feasible (but not strictly feasible), as $\mathcal{B}^{II}$ is non-empty, and the objective value is bounded above by zero. Since \hyperlink{primaltilde}{$(\widetilde{P})$} is strictly feasible, Proposition \ref{prop:2.1} yields $\text{val}(\widetilde{P})=\text{val}(\widetilde{D})=0$ and $\mathcal{S}(\widetilde{D})\neq\emptyset$. That is, there exists some $y^*\in T$ such that $-A^*y^*\in C^*$ and $\langle y^*,b\rangle=0$.
\end{proof}

The proof of the almost equivalence follows.

\begin{proof}[Proof of Theorem 3]
$(i)$ By Lemma \ref{lem:4.2}, either there exists some $x^*\in C\setminus\{0\}$ such that $Ax^*\in \text{int}K^*$, or some $y^*\in K\setminus\{0\}$ such that $-A^*y^*\in \text{int}C^*$ (but not both). Suppose that the former is true, and let $x$ be a feasible solution for \hyperlink{primal}{$(P)$}. Then $x^*+x$ is a strictly feasible solution for \hyperlink{primal}{$(P)$}. The result follows from Proposition \ref{prop:2.1}. Likewise, if the latter is true, then \hyperlink{dual}{$(D)$} is strictly feasible and the result follows from the dual counterpart of Proposition \ref{prop:2.1}.\par
$(ii)$ Suppose that $v=0$ and that the set $\mathcal{B}^{I}\cap c^{\perp}$ is empty. Then, (\ref{eq:4.3}) and Lemma \ref{lem:4.3} imply that \hyperlink{dual}{$(D)$} is strictly feasible. The result once again follows from the dual counterpart of Proposition \ref{prop:2.1}. Accordingly, if $\mathcal{B}^{II}\cap b^{\perp}=\emptyset$, then \hyperlink{primal}{$(P)$} is strictly feasible by Lemma \ref{lem:4.3} and (\ref{eq:4.3}), and strong duality between \hyperlink{primal}{$(P)$} and \hyperlink{dual}{$(D)$} follows from Proposition \ref{prop:2.1}.\par
$(iii)$ Suppose that $v=0$, $\mathcal{B}^{I}\cap c^{\perp}\neq\emptyset$ and $\mathcal{B}^{II}\cap b^{\perp}\neq\emptyset$. Then, (\ref{eq:4.3}) and Lemma \ref{lem:4.3} directly imply that neither of the programs \hyperlink{primal}{$(P)$}, \hyperlink{dual}{$(D)$} is strictly feasible. If, in addition, there exist $x^*\in\mathcal{B}^{I}$ feasible for \hyperlink{primal}{$(P)$} such that $Ax^*\in K^*$, $\langle c,x^*\rangle=0$, and there exists some $y^*\in\mathcal{B}^{II}$ feasible for \hyperlink{dual}{$(D)$} such that $-A^*y^*\in C^*$, $\langle y^*,b\rangle=0$, then the complementarity conditions (\ref{eq:2.3}) hold. Therefore, $\text{val}(P)=\text{val}(D)=0$ and $x^*$, $y^*$ are optimal solutions for \hyperlink{primal}{$(P)$} and \hyperlink{dual}{$(D)$}, respectively.
\end{proof}

 Recall from part $(iii)$ of Theorem \ref{th:4.1} that we do not determine strong duality when either of the sets $\mathcal{B}^{I}\cap c^{\perp}\cap\mathcal{F}(P)$, $\mathcal{B}^{II}\cap b^{\perp}\cap\mathcal{F}(D)$ is empty. We now illustrate that this unexplored case is sharp. More specifically, we fix a fair semidefinite game and construct various primal-dual pairs of SDPs in order to show that the absence of a zero duality gap is possible in each of the following two ``pathological" scenarios: a) exactly one of $\mathcal{B}^{I}\cap c^{\perp}\cap\mathcal{F}(P)$, $\mathcal{B}^{II}\cap b^{\perp}\cap\mathcal{F}(D)$ is empty; and b) both sets are empty.

\begin{example}\label{ex:4.4}
    Consider the closed convex cones $\mathbb{R}^{n}_+\subseteq\mathbb{R}^{n}$, $\mathcal{S}^m_+\subseteq\mathcal{S}^m$, where $\mathcal{S}^m$ is the space of all $m$ by $m$ symmetric matrices with real entries, and $\mathcal{S}^{m}_+$ is the space of all $m$ by $m$ positive semidefinite matrices with real entries. We equip $\mathbb{R}^{n}$ with the standard inner product and $\mathcal{S}^{m}$ with the Frobenius inner product. Recall that, for $B,C\in\mathcal{S}^m$ the Frobenius product is the inner product $\langle B,C\rangle=  B\bullet C:=\text{tr}(BC)=\sum_{i=1}^{m}\sum_{j=1}^{m}b_{ij}c_{ij}$. In addition, it is a known fact that $\left(\mathcal{S}^{m}_{+}\right)^*=\mathcal{S}^{m}_{+}$ and $\left(\mathbb{R}^{n}_{+}\right)^*=\mathbb{R}^{n}_{+}$ for any $m,n\in\mathbb{N}$, i.e., the cones are self-dual. It is also known that the interior of $\mathcal{S}^{m}_{+}$ is exactly the set of all $m$ by $m$ positive definite matrices (Lemma \ref{lem:3.4}).\vspace{0.2cm}\par
 For a matrix $X\in\mathcal{S}^3$ we will write $X=\begin{pmatrix}
        x_{11} & x_{12} & x_{13} \\
        x_{12} & x_{22} & x_{23} \\
        x_{13} & x_{23} & x_{33}
        \end{pmatrix}$, and, using the usual notation (assuming that the dimension in this example remains fixed and equal to 3), we will write $X\succeq 0
        $ when $X\in\mathcal{S}^3_+$. Moreover, we will denote a 3 by 3 matrix with diagonal elements $s_1,s_2,s_3$ and all other entries equal to zero by $\text{diag}(s_1,s_2,s_3)$.\par
        Consider a two-player zero-sum game $G=(S,T,u)$, where
    \begin{equation*}
        S:=  \left\{X\in\mathcal{S}^3_{+}:E\bullet X=1\right\}, \;\;E:=I_3;\;\hspace{0.04cm}
    \end{equation*}
    \begin{equation*}
        T:=  \left\{y\in\mathbb{R}^2_+:e^\top y=1\right\}, \;\;e:=(1,1)^\top;
    \end{equation*}
    \begin{equation*}        A:\mathcal{S}^3\rightarrow\mathbb{R}^2\mbox{\;with\;} A(X):= (A_1\bullet X,A_2\bullet X)^\top,\mbox{\;where\;} A_1:=  \begin{pmatrix}
            0 & -1 & 0 \\
            -1 & 0 & 0 \\
            0 & 0 & 1
        \end{pmatrix}\text{\;and\;} A_2:=  \begin{pmatrix}
            0 & 0 & 0 \\
            0 & -1 & 0 \\
            0 & 0 & 0
        \end{pmatrix}.
    \end{equation*}

The payoff function is then equal to $u(X,y)=A(X)^\top y=y_1(x_{33}-2x_{12})-y_2x_{22}$, for every $X\in S$, $y\in T$. The following strategies constitute a saddle point:
\begin{equation*}
 X^*=\begin{pmatrix}
    1 & 0 & 0 \\
    0 & 0 & 0\\
    0 & 0 & 0
\end{pmatrix},\;y^*=(0,1)^\top.
\end{equation*} 
Further, $u(X^*,y^*)=0$. Indeed, we have $u(X^*,y^*)=0\geq-x_{22}= u(X,y^*)$ for every $X\in S$. Also, $u(X^*,y)\geq u(X^*,y^*)$ for every $y\in T$, since each side of the inequality is zero. In addition, the optimal strategy $y^*$ is unique for player $II$. To see this, let $\hat{y}\in T$ be another optimal strategy . Then, since $v=0$, we must have $u(X,\hat{y})\leq 0$ for every $X\in S$. For the strategy $\hat{X}:= \text{diag}(0,0,1)$ we deduce $\hat{y}_1\leq 0$. As $\mathcal{B}^{II}=\{y^*\}$, a simple inspection then shows that $\mathcal{B}^I=\{X\succeq 0:x_{22}=0,\;x_{11}+x_{33}=1\}$.\vspace{0.2cm}\par
Now, for the above linear operator $A$ and points $b=(b_1,b_2)^\top\in\mathbb{R}^2$, $C\in\mathcal{S}^{3}$ (free for now), we consider the following primal-dual pair of SDPs:

\vspace*{-0.7cm}\begin{center}\hypertarget{pb}{}
\begin{equation*}{({P_b})\;\;\;\;\;\;\;\;}
\begin{array}{ll}
    \inf\limits_{X}\;\; C\bullet X\\
    \text{s.t.}\;\; x_{33}-2x_{12}\geq b_1 \\
    \;\;\;\;\;\;-x_{22} \geq b_2 \\ \;\;\;\;\;\;X\succeq 0
    \end{array}\hspace{1.4cm}
\end{equation*}
\end{center}

\begin{center}\hypertarget{dc}{}
\vspace{-0.3cm}\begin{equation*}{\hspace{1.3cm} ({D_C})\;\;\;\;\;\;\;\;}
\begin{array}{ll}
    \sup\limits_{y}\;\; b^\top y\\
    \text{s.t.}\;\;y_1 \begin{pmatrix}
            0 & 1 & 0 \\
            1 & 0 & 0 \\
            0 & 0 & -1
        \end{pmatrix} + y_2 \begin{pmatrix}
            0 & 0 & 0 \\
            0 & 1 & 0 \\
            0 & 0 & 0
        \end{pmatrix} +C \succeq 0\\ \;\;\;\;\;\;\;y_1,\;y_2\geq 0
    \end{array}
    \end{equation*}
\end{center}

We study four different primal-dual pairs:
\begin{enumerate}
    \item Let $b:=  (1,0)^\top$ and $C:= \text{diag}(0,1,1)$. Then, \hyperlink{pb}{$(P_b)$} is feasible (take $x_{33}\geq 1$ and all other entries of $X$ to be equal to $0$), but not strictly feasible: The second constraint of \hyperlink{pb}{$(P_b)$} implies that, for any feasible $X\succeq 0$, $x_{22}$ has to be zero. The dual \hyperlink{dc}{$(D_C)$} is also feasible (take $y_1=y_2=0$), but not strictly feasible: We can see that any feasible solution must satisfy $y_1=0$ and that the matrix $\text{diag}(0,y_2+1,1)$ is never positive definite for any non-negative value of $y_2$. Observe that the non-strict feasibility of the two programs is also derived from Theorem \ref{th:4.1}. Indeed, we find $\mathcal{B}^{I}\cap C^{\perp}\neq \emptyset$ and  $\mathcal{B}^{II}\cap b^{\perp}\neq\emptyset$, as $C\bullet X^*=0$ and $b^\top y^*=0$, respectively. Further, note that $y^*$ is a feasible solution for \hyperlink{dc}{$(D_C)$}. In addition, $\mathcal{B}^{I}\cap \mathcal{F}(P_b)=\{\hat{X}\}$, for $\hat{X}\in S$ as above. However, $C\bullet\hat{X}\neq 0$. Therefore, we deduce that $\mathcal{B}^{I}\cap C^{\perp}\cap\mathcal{F}(P_b)=\emptyset$ and $\mathcal{B}^{II}\cap b^{\perp}\cap\mathcal{F}(D_C)\neq\emptyset$, which falls in the pathological sub-region not covered by Theorem \ref{th:4.1}. Notice that the value of the primal is attained for $x_{22}=0$, $x_{33}=1$ and it is equal to 1, while the value of the dual is equal to $0$ and attained trivially for $y_1=0$ and any $y_2\geq 0$. Namely, a non-zero duality gap exists. 
\item Let $b=(0,0)^\top$ and $C=A_1$. Then, $\mathcal{B}^{I}\cap C^{\perp}\cap\mathcal{F}(P_b)\neq\emptyset$ and $\mathcal{B}^{II}\cap b^{\perp}\cap \mathcal{F}(D_C)=\emptyset$. This comes from the fact that $C\bullet X^*=0$ and $X^*$ is feasible for \hyperlink{pb}{$(P_b)$}, while $y^*$ is not a feasible solution of \hyperlink{dc}{$(D_C)$}. Again, none of the two programs is strictly feasible by Theorem \ref{th:4.1} $(\mbox{as\;}C\bullet X^*=0$, $b^\top y^*=0)$. In this case, however, we find $\text{val}(P_b)=\text{val}(D_C)=0$, $\mathcal{S}(P_b)\neq\emptyset$ and $\mathcal{S}(D_C)\neq\emptyset$.
\item Let $b=(1,0)^\top$ and $C=\begin{pmatrix}
    0 & -1/2 & 0 \\
    -1/2 & 0 & 0 \\
    0 & 0 & 1
\end{pmatrix}$. Then $\mathcal{B}^{I}\cap C^{\perp}=\{X^*\}$ and $\mathcal{B}^{II}\cap b^{\perp}=\{y^*\}$, so none of the programs is strictly feasible, as implied by Theorem \ref{th:4.1}. Moreover, neither $X^*$ nor $y^*$ are feasible points. Hence, both $\mathcal{B}^{I}\cap C^{\perp}\cap\mathcal{F}(P_b)$ and $\mathcal{B}^{II}\cap b^{\perp}\cap\mathcal{F}(D_C)$ are empty. Observe that every feasible solution for \hyperlink{pb}{$(P_b)$} must satisfy $x_{22}=x_{12}=0$ and $x_{33}\geq 1$, while every feasible one of \hyperlink{dc}{$(D_C)$} must satisfy $y_1=1/2$. Thus, $\text{val}(P_b)=1>1/2=\text{val}(D_C)$, i.e., strong duality fails.
\item Let $b=(1,0)^\top$ and $C=A_1$. Then, as in the previous example case, we find $\mathcal{B}^{I}\cap C^{\perp}=\{X^*\}$, $\mathcal{B}^{II}\cap b^{\perp}=\{y^*\}$ and $\mathcal{B}^{I}\cap C^{\perp}\cap\mathcal{F}(P_b)=\mathcal{B}^{II}\cap b^{\perp}\cap\mathcal{F}(D_C)=\emptyset$. In this example, however, strong duality holds as $\text{val}(P_b)=1=\text{val}(D_C)$, and both optimal values are attained.
\end{enumerate}
\vspace{-0.2cm}

Finally, the second part of $(iii)$ of Theorem \ref{th:4.1} is elucidated: Let $b=(0,0)^\top$ and $C=\text{diag}(0,1,1)$. Then both $X^*$ and $y^*$ are feasible solutions, the complementarity conditions (\ref{eq:2.3}) hold and no duality gap exists.\vspace{0.5cm}
\end{example}

We conclude this section with a game-dependent criterion of infeasibility for the primal-dual pair. This result extends on those of \cite{dantzig1951proof, ickstadt2024semidefinite} (see $(iii)$ of Proposition \ref{prop:ickstadt} below). Its proof is omitted, as it can be found in the first paragraph of the proof of Lemma \ref{lem:4.3}.\vspace{0.2cm}\par

\begin{proposition}\label{prop:infeasibility}
    Consider a linear operator $A:X\rightarrow Y^*$, closed convex cones $C\subseteq X,\;K\subseteq Y$, and vectors $b\in Z$, $c\in W$. Also, consider the game $G=(\alpha,\beta,A)$ for some $\alpha\in\text{int} C^*,\;\beta\in\text{int}K^*$, and assume that $v=0$. If there exists $x^*\in\mathcal{B}^{I}$ such that $\langle c,x^*\rangle<0$ or $y^*\in\mathcal{B}^{II}$ such that $\langle y^*,b\rangle>0$, then at least one of $(P)$, $(D)$ is infeasible.
\end{proposition}

We ought to remark that, in contrast to the two preceding works, one needs only study the case $v=0$ for any chosen game in order to determine infeasibility of the primal-dual pair. This comes from the fact that $v\neq0$ implies strict feasibility of either program according to the proof of part $(i)$ of Theorem \ref{th:4.1}. As a matter of fact, while feasibility is used in the proof of part $(i)$, it is not actually needed. To see this, suppose without loss of generality that (according to Lemma \ref{lem:4.2}) there exists $x^*\in C\setminus\{0\}$ such that $Ax^*\in\text{int}K^*$. Then, there exists $\varepsilon>0$ small enough such that $Ax^*-\varepsilon b\in\text{int}K^*$, or equivalently, there exists $\kappa>0$ sufficiently large so that $A(\kappa x^*)-b\in \text{int}K^*$ and $\kappa x^*\in C$.

\vspace{-0.1cm}\subsection{The almost equivalence in the semidefinite setting}\vspace{-0.2cm}\label{sec:SDPgames}

In this subsection we study the almost equivalence of Theorem \ref{th:4.1} in the special setting of semidefinite games, and give results complementary to the almost equivalence that appears in Ickstadt et al. \cite{ickstadt2024semidefinite}. As already argued in the introduction, the two results are of different nature, but Theorem \ref{th:4.1} can be considered stronger in the sense that we give a game-dependent characterization of strict feasibility and cover SDP primal-dual pairs where a zero-duality gap exists but at least one program has no optimal solutions, all by studying any user-chosen game, in contrast to \cite{dantzig1951proof, ickstadt2024semidefinite} where the authors study a special symmetric game. \par

More specifically, we characterize a sub-case of the exception that appears in the almost equivalence of \cite{ickstadt2024semidefinite} and we claim that Theorem \ref{th:4.1} proves strong duality in cases that are not covered by the latter. We then use these results to strengthen the aforementioned work by providing an improved scheme for determining strong duality for a primal-dual pair of semidefinite programs.\par

Recall the pair \hyperlink{primaldual}{$\{(P),(D)\}$} and set $C=\mathcal{S}^m_{+}$, $K=\mathbb{R}^n_{+}$, $A:\mathcal{S}^m\to\mathbb{R}^n$ with $A=(A_1,...,A_n)$, where $A_i\in\mathcal{S}^m$ for every $i\in\{1,...,n\}$. Also, let $C\in\mathcal{S}^m$ and $b\in\mathbb{R}^n$. The primal-dual pair then becomes
\hypertarget{primaldualsdp}{} \hypertarget{primalsdp}{} \hypertarget{dualsdp}{}
\vspace{-0.7cm}\begin{center}
\begin{equation*}{\;\;\;\;\;\;\;\;\;(P-SDP)\;\;\;\;\;\;\;\;\;\;\;}
\begin{array}{ll}
    \inf\limits_{X\in\mathcal{S}^m}\;\; C\bullet X\\
    \hspace{0.2cm}\text{s.t.}\;\;\;\; A_i\bullet X\geq b_i,\;\; i=1,...,n\\ \;\;\;\;\;\;\;\;\;\;\hspace{0.03cm}X\succeq 0
    \end{array}\;\;\;\;\;\;
\end{equation*}
\end{center}
\vspace{-0.7 cm}\begin{center}
\begin{equation*}{(D-SDP)\;\;\;\;\;\;\;\;\;\;\;\;}
\begin{array}{ll}
    \sup\limits_{y\in\mathbb{R}^n}\;\; b^\top y\\
    \hspace{0.05cm}\text{s.t.}\;\;\; -\sum\limits_{j=1}^{n} y_j A_j + C\succeq 0\\ \;\;\;\;\;\;\;\;\; y\geq 0
    \end{array}\;\;\;\;\;\;
    \end{equation*}
\end{center}

In addition, consider the feasibility system\vspace{-0.2cm}

\begin{subequations}\label{eq:4.7}
\begin{align}
A_i\bullet \bar{X}-b_i \bar{t}\geq 0,\;\;\;  i=1,...,n,\label{eq:4.7a}\\
-\sum\limits_{j=1}^{n}\bar{y}_j A_j+\bar{t} C\succeq 0,\label{eq:4.7b}\hspace{0.7cm}\\
b^\top \bar{y}-C\bullet \bar{X}\geq 0,\label{eq:4.7c}\hspace{0.7cm}\\
e_n^\top \bar{y}+ \text{tr}(\bar{X})+\bar{t}=1,\label{eq:4.7d}\hspace{0.7cm}\\
\bar{y}\in \mathbb{R}^{n}_{+},\;\;\bar{X}\in\mathcal{S}^{m}_{+},\;\;\bar{t}\geq 0.\label{eq:4.7e}\hspace{0.7cm}
\end{align}
\end{subequations}

By invoking a symmetric version of the minimax theorem, Ickstadt et al. \cite[Lemma 5.2]{ickstadt2024semidefinite} showed that the above system has at least one solution. Using this, they proved a natural semidefinite generalization of Dantzig's \cite{dantzig1951proof} initial result of the almost equivalence:\vspace{0.2cm}\par
\begin{proposition}[\normalfont{Theorem 5.3 \cite{ickstadt2024semidefinite}}]\label{prop:ickstadt} Let $(\bar{y},\bar{X},\bar{t})$ be a solution of the system (\ref{eq:4.7}). Then,\vspace{-0.2cm}
    \begin{enumerate}[(i)]
        \item $\bar{t}(b^\top \bar{y}-C\bullet \bar{X})=0$.\vspace{-0.1cm}
        \item If $\bar{t}>0$, then $\bar{X}\bar{t}^{-1}$ and $\bar{y}\bar{t}^{-1}$ are optimal solutions to the pair \hyperlink{primaldualsdp}{$\{(P-SDP),(D-SDP)\}$}.\vspace{-0.1cm}
        \item If $b^\top \bar{y}- C\bullet \bar{X}>0$, then either \hyperlink{primalsdp}{$(P-SDP)$} or \hyperlink{dualsdp}{$(D-SDP)$} is infeasible.
    \end{enumerate}
\end{proposition}

In comparison with Theorem \ref{th:4.1},  Proposition \ref{prop:ickstadt} makes no mention of strict feasibility and it implies strong duality only with both $\mathcal{S}(P)$ and $\mathcal{S}(D)$ being non-empty. However, the case where a zero duality gap exists and at least one program does not attain its optimal value, does not represent a rare event in semidefinite programming \cite[Example 4.1.1]{wolkowicz2012handbook}. Moreover, in similar fashion to the linear programming setting in \cite{dantzig1951proof}, the above result ignores the case ``$\bar{t}=0$ and $b^\top y- C\bullet \bar{X}=0$". To this end, consider all solutions $(\bar{y},\bar{X},\bar{t})$ of (\ref{eq:4.7}) that lie in the following pathological sub-region:\vspace{0.2cm}\par
\hypertarget{theta1}{}
\noindent\textbf{($\Theta 1$):} $\bar{y}\neq 0$, $\bar{X}\neq 0_m$, $\bar{t}=0$ and $b^\top \bar{y}= C\bullet \bar{X}$.\vspace{0.2cm}\par

It turns out that the existence of such solutions amounts to lack of strict feasibility for the primal-dual pair.

\begin{lemma}\label{lem:theta1}
    Suppose that the primal-dual pair \hyperlink{primaldualsdp}{$\{(P-SDP),(D-SDP)\}$} is feasible. Also consider the game $G=(\alpha,\beta,A)$ for $\alpha\in\text{int}\mathcal{S}^{m}_{+}$ and $\beta\in\text{int}\mathbb{R}^{n}_{+}$, and let $v$ be its game value and $\mathcal{B}^I\times\mathcal{B}^{II}$ the set of its Nash equilibria. Then, there exists a solution $(\bar{y},\bar{X},\bar{t})$ to (\ref{eq:4.7}) that satisfies \hyperlink{theta1}{$(\Theta1)$} if and only if $v=0$, $\mathcal{B}^I\cap C^{\perp}\neq\emptyset$ and $\mathcal{B}^{II}\cap b^{\perp}\neq\emptyset$ for every $(\alpha,\beta)\in \text{int}\mathcal{S}^{m}_{+}\times\text{int} \mathbb{R}^{n}_{+}$.
\end{lemma}

\begin{proof}
    We first prove the ``if" direction. Consider the game $G=(I_m,e_n,A)$. Suppose $v=0$ and let $(X^*,y^*)$ be a Nash equilibrium so that $C\bullet X^*=0$ and $b^\top y^*=0$. By (\ref{eq:4.3}) and Lemma \ref{lem:3.5} we deduce that $A_i\bullet X^*\geq 0$ $\forall i\in\{1,...,n\}$ and $-\sum\limits_{j=1}^{n} y^*_j A_j\succeq 0$. Hence, $(y^*/2,X^*/2,0)$ solves (\ref{eq:4.7}).\par 
    Conversely, suppose that $(\bar{y},\bar{X},\bar{t})$ is a solution to (\ref{eq:4.7})  with $\bar{y}\neq 0$ and $\bar{X}\neq 0_m$ so that $\bar{t}=0$ and $b^\top \bar{y}-C\bullet \bar{X}=0$. Let $\alpha$ be a positive definite matrix and $\beta$ be a positive real vector. Then, by Lemma \ref{lem:3.4} we have $\text{tr}(\alpha \bar{X})>0$ and $\beta^\top \bar{y}>0$. Hence, $\left(\bar{X}(\text{tr}(\alpha \bar{X}))^{-1},\bar{y} (\beta^\top\bar{y})^{-1}\right)$ is a saddle point for $G=(\alpha,\beta,A)$, and $v=0$, as (\ref{eq:4.7a})-(\ref{eq:4.7b}) and Lemma \ref{lem:3.5} imply (\ref{eq:4.3}). The fact that $b^\top\bar{y}=C\bullet \bar{X}=0$ follows from the first part of the proof of Lemma \ref{lem:4.3}: Due to the feasibility of the SDP pair, we must have $C\bullet \bar{X}\geq 0$ and $b^\top \bar{y}\leq 0$. As $b^\top \bar{y}= C\bullet \bar{X}$, both inequalities must hold at equality.
\end{proof}

As already mentioned, \cite[Lemma 5.2]{ickstadt2024semidefinite} implies that the system (\ref{eq:4.7}) always has a solution. At the same time, by Theorem \ref{th:3.3} exactly one of the three preconditions of Theorem \ref{th:4.1} holds for the game $G=(\alpha,\beta,A)$, for any interior points $\alpha,\beta$. Therefore, from Lemma \ref{lem:theta1} and Theorems \ref{th:3.3}, \ref{th:4.1} we deduce the following:

\begin{theorem}\label{th:SDPsubcase}
    Suppose that the primal-dual pair \hyperlink{primaldualsdp}{$\{(P-SDP),(D-SDP)\}$} is feasible. If no solution of (\ref{eq:4.7}) satisfies \hyperlink{theta1}{$(\Theta1)$}, then at least one of \hyperlink{primalsdp}{$(P-SDP)$}, \hyperlink{dualsdp}{$(D-SDP)$} is strictly feasible. Moreover, $\text{val}(P-SDP)=\text{val}(D-SDP)$ and at least one of $\mathcal{S}(P-SDP)$, $\mathcal{S}(D-SDP)$ is non-empty, convex and bounded.
\end{theorem}

 As a consequence, the almost equivalence of the present work is able to guarantee strong duality in cases where Proposition \ref{prop:ickstadt} fails to do so. For instance, Theorem \ref{th:4.1} (Theorem \ref{th:SDPsubcase}) covers the case where all solutions of (\ref{eq:4.7}) are of the form $(0,\bar{X},0)$. We give a concrete example.

\begin{example}\label{ex:example2}
    Let $m=3$, $n=2$ and consider the vector $b:= (1,0)^\top$ and the symmetric matrices 
\begin{equation*}
    A_1:= \begin{pmatrix}
    0 & 1 & 0  \\
    1 & 0 & 0\\
    0 & 0 & 0
\end{pmatrix},\;\;
A_2:= \begin{pmatrix}
    0 & 0 & 0  \\
    0 & 1 & 0\\
    0 & 0 & -1
\end{pmatrix},\;\;
C:= \begin{pmatrix}
    1 & 0 & 0  \\
    0 & 0 & 0\\
    0 & 0 & 1
\end{pmatrix}.
\end{equation*}
\end{example}

Consider the system (\ref{eq:4.7}) and let $\bar{X}\in\mathcal{S}^{3}_{+}$, $\bar{y}\in\mathbb{R}^{2}_{+}$ and $\bar{t}\geq 0$ constitute some feasible solution. The constraint (\ref{eq:4.7b}) gives $\bar{y}_1=\bar{y}_2=0$. In turn, the conditions (\ref{eq:4.7c}) and (\ref{eq:4.7e}) give $\bar{x}_{11}=\bar{x}_{33}=0$. So (\ref{eq:4.7a}) yields $\bar{x}_{22}\geq 0$ and $\bar{t}\leq 0$. Therefore, the system (\ref{eq:4.7}) has the unique solution
\begin{equation}\label{eq:solution}
    \bar{y}=(0,0)^\top,\;\;\bar{X}=\begin{pmatrix}
    0 & 0 & 0 \\
    0 & 1 & 0 \\
    0 & 0 & 0
\end{pmatrix},\;\;\bar{t}=0.
\end{equation}
Observe that the constraint (\ref{eq:4.7c}) is tight, hence (\ref{eq:solution}) falls in the pathological region not covered by Proposition \ref{prop:ickstadt}, and one can therefore not determine strong duality using the latter result. However, we can invoke Theorem \ref{th:SDPsubcase} in order to deduce strong duality. The SDP primal-dual pair becomes
\vspace*{-0.7cm}\begin{center}
\begin{equation*}{(P-SDP)\;\;\;\;\;\;\;\;\;\;}
\begin{array}{ll}
    \inf\limits_{X\in\mathcal{S}^3}\;\; x_{11}+x_{33}\\
    \hspace{0.15cm}\text{s.t.}\;\;\;\; 2x_{12}\geq 1\\
    \;\;\;\;\;\;\;\;\;\hspace{0.1cm}x_{22}\geq x_{33}\\
    \;\;\;\;\;\;\;\;\;\hspace{0.04cm}X\succeq 0
    \end{array}\;\;\;\;\;\;\;\;\;\;\;\;\;
\end{equation*}
\end{center}
\vspace{-0.7 cm}\begin{center}
\begin{equation*}{\;\;\;\;\;\;\;\;\;\;\;\;\;\;\;\;\;\;\;\;\;\;\;\;\;(D-SDP)\;\;\;\;\;\;\;\;\;\;}
\begin{array}{ll}
    \sup\limits_{y\in\mathbb{R}^2}\;\; y_1\\
    \hspace{0.05cm}\text{s.t.}\;\; \begin{pmatrix}
    1 & -y_1 & 0 \\
    -y_1 & -y_2 & 0 \\
    0 & 0 & y_2+1
\end{pmatrix}\succeq 0\\ \;\;\;\;\;\;\;\;\;\hspace{0.04cm} y\geq 0
    \end{array}\;\;\;\;\;\;\;\;\;\;\;\;\;
    \end{equation*}
\end{center}
Observe that both problems are feasible (in fact the dual has the unique feasible solution $y^*=0$). Therefore, from Theorem \ref{th:SDPsubcase} we deduce that at least one of the programs is strictly feasible and the two programs have the same optimal value. Indeed, \hyperlink{primalsdp}{$(P-SDP)$} is strictly feasible and $\text{val}(P-SDP)=\text{val}(D-SDP)=0$: For small $\varepsilon>0$ set $x_{11}=x_{33}=\varepsilon$, $x_{12}=1/2$, $x_{22}=1/\varepsilon$ and all other entries to $0$; then let $\varepsilon\to 0$. Importantly, we find $\mathcal{S}(P-SDP)=\emptyset$ and $\mathcal{S}(D-SDP)=\left\{(0,0)^\top\right\}$, which also verifies our fourth argument that appears in Section \ref{sec:relatedwork}.\par

Unfortunately, one cannot expect more from the almost equivalence of Theorem \ref{th:4.1} in the case \hyperlink{theta1}{$(\Theta1)$}. To be more specific, consider a semidefinite game $G=(\alpha,\beta,A)$ and let $v$ be its game value and $\mathcal{B}^I\times\mathcal{B}^{II}$ be the set of its Nash equilibria. In turn, consider the following scenario: \vspace{0.2cm}\par
\hypertarget{theta2}{}
\noindent\textbf{($\Theta 2$):} $v=0$, $\mathcal{B}^{I}\cap C^{\perp}\neq\emptyset$, $\mathcal{B}^{II}\cap b^{\perp}\neq\emptyset$, and either $\mathcal{B}^{I}\cap C^{\perp}\cap\mathcal{F}(P-SDP)=\emptyset$ or $\mathcal{B}^{II}\cap b^{\perp}\cap\mathcal{F}(D-SDP)=\emptyset$.\vspace{0.2cm}\par

Theorem \ref{th:4.1} and Example \ref{ex:4.4} say that the only possible case where strong duality may fail is the one where \hyperlink{theta2}{$(\Theta2)$} holds for all pairs $(\alpha,\beta)\in \text{int}\mathcal{S}^{m}_{+}\times\text{int} \mathbb{R}^{n}_{+}$. Using this, we show that \hyperlink{theta1}{$(\Theta1)$} is ``sharp" when all solutions of (\ref{eq:4.7}) fall within the region not covered by Proposition \ref{prop:ickstadt}.

\begin{theorem}\label{th:theta2}
     Suppose that the primal-dual pair \hyperlink{primaldualsdp}{$\{(P-SDP),(D-SDP)\}$} is feasible. Also consider the game $G=(\alpha,\beta,A)$ for $\alpha\in\text{int}\mathcal{S}^{m}_{+}$ and $\beta\in\text{int}\mathbb{R}^{n}_{+}$, and let $v$ be its game value and $\mathcal{B}^I\times\mathcal{B}^{II}$ the set of its Nash equilibria. If there exists a solution to (\ref{eq:4.7}) that satisfies \hyperlink{theta1}{$(\Theta1)$} and no solutions with $\bar{t}>0$ of the same system exist, then \hyperlink{theta2}{$(\Theta2)$} holds for every $(\alpha,\beta)\in \text{int}\mathcal{S}^{m}_{+}\times\text{int} \mathbb{R}^{n}_{+}$.
\end{theorem}

\begin{proof}
    Consider a triple $(\bar{y},\bar{X},\bar{t})$ that satisfies (\ref{eq:4.7}) and \hyperlink{theta1}{$(\Theta1)$}. Lemma \ref{lem:theta1} directly gives the first part of \hyperlink{theta2}{$(\Theta2)$} for all $(\alpha,\beta)\in\text{int}\mathcal{S}^{m}_{+}\times \text{int}\mathbb{R}^{n}_{+}$. For the second part, suppose we can find some pair $(\alpha,\beta)\in \text{int}\mathcal{S}^{n}_{+}\times\text{int} \mathbb{R}^{m}_{+}$ so that the game $G=(\alpha,\beta,A)$ has value equal to zero and there exist $X^*\in\mathcal{B}^{I}\cap C^{\perp}\cap\mathcal{F}(P-SDP)$ and $y^*\in\mathcal{B}^{II}\cap b^{\perp}\cap\mathcal{F}(D-SDP)$. Then, the triplet $(y^*,X^*,1)$ satisfies (\ref{eq:4.7a})-(\ref{eq:4.7c}) and (\ref{eq:4.7e}). Further, suppose that $e_m^\top y^*+\text{tr}(X^*)+1=\mu$, where $\mu$ is positive due to Lemma \ref{lem:3.4}. Then, $\mu^{-1}(y^*,X^*,1)$ solves (\ref{eq:4.7}), a contradiction.
\end{proof}

The above analysis reinforces the result of Ickstadt et al. \cite{ickstadt2024semidefinite} and provides an improved \textit{plan} on determining strict feasibility and a zero-duality gap for the pair \hyperlink{primaldualsdp}{$\{(P-SDP),(D-SDP)\}$}: Suppose feasibility of the latter and let $(\bar{y},\bar{X},\bar{t})$ be a solution to (\ref{eq:4.7}). If $\bar{t}>0$, then strong duality holds. If no such solution exists, then the next step is to check whether a solution that satisfies \hyperlink{theta1}{$(\Theta1)$} exists. If one doesn't, then at least one program is strictly feasible, thus there is a zero-duality gap. If, however, a solution that satisfies \hyperlink{theta1}{$(\Theta1)$} does exist, then \hyperlink{theta2}{$(\Theta2)$} holds, meaning we cannot verify strong duality for the primal-dual pair.\vspace{0.2cm}\par

\vspace{-0.1cm} \subsection{A generalization of Ville's theorem of the alternative}\vspace{-0.2cm}
 
 We next prove that the minimax theorem for bases of convex cones, Theorem \ref{th:3.3}, is equivalent to a theorem of the alternative, which can be regarded as an extension of Ville's theorem  to reflexive Banach spaces. Results that deal with the consistency of alternative systems of the subsequent form have already been given for more general operators in Banach spaces; see, for instance, \cite[Theorem 4.1]{doi:10.1080/02331938908843428} and \cite[Theorem 3.1]{yang2000theorems}. However, none of the previous generalized theorems of the alternative have been known to be connected to the minimax theory of games, except the classical theorem of Ville \cite[Theorem 11]{adler2013equivalence}, \cite[Proposition 4]{von2024zero}. We are still working under the topological setting of Subsection \ref{sec:3.2}. Consider the following pairs of alternative systems:\vspace{0.2cm}
\begin{subequations}
\begin{align}
\exists y\in K\setminus\{0\} \text{\;\;such that\;} -A^*y\in C^*\label{eq:19a}\\
\exists x\in C\setminus\{0\} \text{\;\;such that\;} Ax\in \text{int}K^*.\label{eq:19b}
\end{align}
\end{subequations}
\begin{subequations}\label{eq:20}\vspace{-0.8cm}
\begin{align}
\exists y\in K\setminus\{0\} \text{\;\;such that\;} -A^*y\in\text{int} C^*\label{eq:20a}\;\;\\
\hspace{0.7cm}\exists x\in C\setminus\{0\} \text{\;\;such that\;} Ax\in K^*.\;\;\;\;\;\;\;\;\;\;\;\label{eq:20b}
\end{align}
\end{subequations}
\vspace*{0.01cm}

\begin{theorem}\label{th:4.5}
   The following are equivalent:\vspace{-0.2cm}
    \begin{enumerate}[(i)]
        \item For any two-player zero-sum game $G=(\alpha,\beta,A)$ defined by $\hyperlink{f1}{(F1)}$-$\hyperlink{f3}{(F3)}$, functionals $\alpha\in\text{int}C^*,\;\beta\in\text{int}K^*$, and a linear operator $A:X\rightarrow Y^*$, the minimax equality (\ref{eq:1.3}) holds.\vspace{-0.1cm}
        \item For any linear operator $A:X\rightarrow Y^*$ exactly one of (\ref{eq:19a})-(\ref{eq:19b}) and exactly one of (\ref{eq:20a})-(\ref{eq:20b}) hold.
    \end{enumerate}
\end{theorem}
\begin{proof}
    The proof of the ``only if" direction is simple, as it follows from the proof of Lemma \ref{lem:4.2} with some minor modifications. For completeness we present a different one, which appears to be shorter. Consider a linear operator $A:X\rightarrow Y^*$ and the game $G=(\alpha,\beta,A)$ for some $(\alpha,\beta)\in\text{int}C^*\times\text{int}K^*$. First, observe that if (\ref{eq:19a}) and (\ref{eq:19b}) ((\ref{eq:20a}) and (\ref{eq:20b}), resp.) are both true at the same time, i.e., if such $x\in C\setminus\{0\}$, $y\in K\setminus\{0\}$ exist, then (\ref{eq:4.2}) holds for those $x$ and $y$, where now the left (right, resp.) inequality is no longer strict. This is a contradiction. It then suffices to consider three cases regarding the value of the game. If $v=0$, then (\ref{eq:4.3}) holds, thus (\ref{eq:19a}) and (\ref{eq:20b}) are true. If $v>0$, then there exists an optimal strategy $x^*\in S$ such that $\langle y,Ax^*\rangle\geq v>0$ $\forall y\in T$. Lemmas \ref{lem:3.4}, \ref{lem:3.5} imply that (\ref{eq:19b}) holds, hence (\ref{eq:20b}) is also true. Similarly, if $v<0$ then there exists $y^*\in T$ such that $-A^*y^*\in\text{int}C^*$, hence (\ref{eq:19a}) and (\ref{eq:20a}) hold.\par
     For the ``if" direction, consider the game $G_A=(\alpha,\beta,A)$ for some linear operator $A$ and some interior points $\alpha$, $\beta$ of $C^*$, $K^*$, respectively. Lemma \ref{lem:3.6} allows us to assume that $\underline{v}>0$ without loss of generality. Define the linear operator $B:X\rightarrow Y^*$ by $B(x):= A(x)-\overline{v}E(x)$, where $E$ is the operator of Lemma \ref{lem:3.6}. We first show that there exists a strategy $\hat{y}\in T$ such that $-B^*\hat{y}\in C^*$. Suppose that this is not the case. Then (\ref{eq:19b}) is true, that is, there exists $x\in C\setminus\{0\}$ such that $Bx\in \text{int}K^*$. We can then find $\varepsilon>0$ sufficiently small so that $Bx-\varepsilon\overline{v}\beta\langle\alpha,x\rangle\in K^*$. But a simple scaling procedure and Lemma \ref{lem:3.5} imply that there exists some $x^*\in S$ such that $\langle y,Ax^*\rangle\geq\overline{v}(1+\varepsilon)$ $\forall y\in T$, hence $\underline{v}\geq\overline{v}(1+\varepsilon)$, which is a contradiction.\par
    We next show that there exists a strategy $\hat{x}\in S$ such that $B\hat{x}\in K^*$. Suppose that this is not the case. Then (\ref{eq:20a}) holds, i.e., there exists some $y\in K\setminus\{0\}$ such that $-B^*y\in \text{int}C^*$. Again, by a trivial scaling procedure, there exists $y^*\in T$ such that $-A^*y^*+\overline{v}\alpha\in\text{int}C^*$. We can find $\varepsilon'\in(0,1)$ such that $-A^*y^*+\overline{v}\alpha(1-\varepsilon')\in C^*$. The dual counterpart of Lemma \ref{lem:3.5} yields $\overline{v}(1-\varepsilon')\geq\langle A^*y^*,x\rangle\;\forall x\in S$, hence $\overline{v}(1-\varepsilon')\geq\overline{v}$, which is a contradiction.\par
    We have thus shown that there exist $\hat{y}\in T$ such that $-B^*\hat{y}\in C^*$ and $\hat{x}\in S$ such that $B\hat{x}\in K^*$. Thereby (\ref{eq:4.3}) holds, which directly implies that the game $G_B= (\alpha,\beta,B)$ has value zero and $(\hat{x},\hat{y})$ is a Nash equilibrium. By Lemma  \ref{lem:3.6} we deduce that $\underline{v}=\overline{v}$ and $(\hat{x},\hat{y})$ is a saddle point for $G_A$.
\end{proof}

\begin{remark} An immediate consequence of Theorem \ref{th:4.5} is the identification of a different proof of minimax (Theorem \ref{th:3.3}), which does not rely on duality theory of conic optimization but on separating hyperplane theorems instead, as $(ii)$ becomes a special case of main results in
\cite{doi:10.1080/02331938908843428, yang2000theorems}. Furthermore, it is easy to observe that, given a linear operator $A$, if at least one of (\ref{eq:19b}), (\ref{eq:20a}) holds, then at least one of \hyperlink{primal}{$(P)$}, \hyperlink{dual}{$(D)$} is strictly feasible, hence Proposition \ref{prop:2.1} implies strong duality for the pair \hyperlink{primaldual}{$\{(P),(D)\}$}. However, if (\ref{eq:19a}) and (\ref{eq:20b}) hold at the same time, then the value of the duality gap ``$\text{val}(P)-\text{val}(D)$" remains undetermined. This observation also justifies the need of stronger results (here Lemmas \ref{lem:4.2}, \ref{lem:4.3}) for a concrete relation between the minimax theorem and strong duality of conic linear programming.
\end{remark}

\vspace{-0.1cm}\section{Applications and Examples}\label{sec:5}\vspace{-0.2cm}

Many two-player zero-sum games that appear in the literature are embedded in the model described by \hyperlink{f1}{(F1)}--\hyperlink{f3}{(F3)}. We present here a series of examples and applications of games that belong to this class, many of which are known in less general strategic setups -- in particular the strategy sets that appear are usually bases or basic cone-leveled sets -- to point out the variety of games covered and to illustrate both the theoretical and the computational purposes of the present work. It is seen that the existential and computational results regarding Nash equilibria of the subsequent two-player zero-sum games are deduced by simply invoking Theorems \ref{th:3.1} and \ref{th:3.3}.

\vspace{-0.1cm}\subsection{The classical two-player zero-sum game}\vspace{-0.2cm}

Consider the classical two-player zero-sum game (von Neumann \cite{Neumann1928}) where $X=W=\mathbb{R}^{m}$, $Z=Y=\mathbb{R}^{n}$, both equipped with the standard inner product $\langle a,b\rangle=a^{\top}b$. Also, let $A$ represent an $m\times n$ real matrix, and the payoff function defined by $u(x,y):= x^{\top}Ay$. The strategy sets of the row and column player, respectively, are defined by 
\begin{equation*}
    S:= \left\{x\in\mathbb{R}^{m}_{+}:\sum_{i=1}^{m}x_i=1\right\} \mbox{\;\;and\;\;} T:= \left\{y\in\mathbb{R}^{n}_{+}:\sum_{j=1}^{n}y_j=1\right\},
\end{equation*}
 which represent the standard m and n-simplex. These sets are bases of the non-negative orthants $C=\mathbb{R}^{m}_{+}$, $K=\mathbb{R}^{n}_{+}$, with $\alpha:= (1,1,...,1)^\top\in\mathbb{R}^{m}$ and $\beta:= (1,1,...,1)^\top\in\mathbb{R}^{n}$, respectively. This two-player zero-sum game can be represented and efficiently solved by a pair of LPs of the form \hyperlink{LPS}{$\{(P-LP),(D-LP)\}$} (see \cite{dantzig2016linear, karmarkar1984new, nesterov1994interior}).

\vspace{-0.1cm}\subsection{Mixed extension of a continuous game with compact strategy sets}\vspace{-0.2cm}

  Consider a two-player zero-sum game $G=(V,U,p)$ where the strategy set of player $I$ is a compact Hausdorff topological vector space $V$, the strategy set of player $II$ is a compact Hausdorff topological vector space $U$, and $p(x,y)$ is a continuous function over $V\times U$. Now let $W:=  C(V)$, $X:=  C(V)^*$, $Z:=  C(U)$, $Y:=  C(U)^*$, where $C(V)$ is the space of all continuous functions over $V$, and $C(V)^*$ is its continuous dual. The pair $(X,W)$ (similarly the pair $(Z,Y)$) is equipped with a bilinear form $\langle\cdot,\cdot\rangle:W\times X\rightarrow \mathbb{R}$ defined by $\langle a,b\rangle:=  \int_{V}a(x)db(x)$. The mixed extension of $G$ is defined as a two-player zero-sum game $G'=(S,T,u)$ with the following structure (see Ville \cite{ville1938theorie}, Glicksberg \cite{glicksberg1952further}, Fan \cite{fan1952fixed}): The strategy set $S$ of player $I$ is the set of all regular Borel probability measures on $V$, namely $S=M(V)$. Note that, with respect to the $w^*$ topology, $M(V)$ is a compact convex subset of $C(V)^*$. Similarly, the strategy set $T$ of player $II$ is the set of all Borel probability measures on $U$, i.e., $T=M(U)$, which is a compact convex subset of $C(U)^*$. Consider the linear operator $A:X\rightarrow Z$ defined by $A(\mu):=  \int_{V}p(x,\cdot)d\mu(x)$, and its adjoint $A^*:Y\rightarrow W$, defined by $A^*(\nu):=  \int_{U}p(\cdot,y)d\nu(y)$ (both continuous w.r.t. the corresponding weak topologies). From the above, and the Fubini theorem, the payoff function of $G'$ is defined by
 \begin{equation*}
     u(\mu,\nu):=  \left\langle \nu,A(\mu)\right\rangle=\iint\limits_{V\times U}p(x,y)d\mu(x)d\nu(y).
 \end{equation*}
 Observe that the strategy sets are bases generated by the convex cones $C:=  M(V)_{+}$, $K:=  M(U)_{+}$, where $M(V)_{+}$ ($M(U)_{+}$ resp.) is the set of all non-negative Borel measures on $V$ ($U$ resp.). Their dual cones are the sets $C(V)_{+}$ and $C(U)_{+}$, the cones of all non-negative functions over $C(V)$ and $C(U)$, respectively. For this two-player zero-sum game, the primal-dual pair \hyperlink{newprimaldual}{$\{(P_\beta),(D_\alpha)\}$} is equivalent to a pair of general capacity problems (see Anderson et al. \cite{anderson1989capacity}).

\vspace{-0.1cm}\subsection{Semi-infinite games}\vspace{-0.2cm} 

 Suppose that one of the players, say $I$, initially has an infinite set of pure strategies. Let $X=\mathbb{R}^{(\mathbb{N})}$ be the space of all real sequences with finite support. Also, let $W=\mathbb{R}^{\mathbb{N}}$ be the space of all real sequences. Spaces $X$ and $W$ are paired with the bilinear function $\langle w,x\rangle:= \sum_{i=1}^{\infty}x_iw_i$. Let $Z=Y=\mathbb{R}^{n}$ for some $n\in\mathbb{N}$, equipped with the standard inner product $\langle y,z\rangle:= y^{\top}z$. The strategy sets $S$, $T$ are given by all mixed strategies of the players:
\begin{equation*}
    S:= \left\{x\in\mathbb{R}^{(\mathbb{N})}:x_i\geq0\;\;\forall i\in\mathbb{N},\;\sum_{i=1}^{\infty}x_i=1\right\}\;\;\text{and}\;\;T:= \left\{y\in\mathbb{R}^{n}_{+}:\sum_{j=1}^{n}y_j=1\right\}.
\end{equation*}
These strategy sets are basic cone-leveled sets defined by the convex cones
\begin{equation*}
    C:= \left\{x\in\mathbb{R}^{(\mathbb{N})}:x_i\geq0\;\mbox{for\;}i=1,2,...\right\},\;\;K:=\mathbb{R}^{n}_{+},
\end{equation*}
and normal vectors $\alpha:= (1,1,...)^\top\in\mathbb{R}^{\mathbb{N}}$ and $\beta:= (1,1,...,1)^\top\in\mathbb{R}^n$.
Here, the payoff function is defined over $S\times T$ by a payoff matrix $A$ with infinite rows and $n$ columns as follows:
\begin{equation*}
    u(x,y):= \langle y, A^{\top}x\rangle=\langle Ay,x\rangle=\sum_{i=1}^{\infty}\sum_{j=1}^{n}x_ia_{ij}y_j.
\end{equation*}
Such games are called semi-infinite games (Soyster \cite{soyster1975semi}) and can be equivalently described by a pair of linear SIP problems, for which various solution methods are known (see Shapiro \cite{shapiro2009semi} for a survey).

\vspace{-0.1cm}\subsection{Semidefinite and non-interactive Quantum games} \vspace{-0.2cm}

 We describe here three classes of generalized two-player zero-sum semidefinite games. We first consider the class of games that appears in Section \ref{sec:SDPgames}. Let $X=W=\mathcal{S}^{m}$ be the space of all $m\times m$ real symmetric matrices, equipped with the Frobenius inner product. Also let $Y=Z=\mathbb{R}^{n}$ equipped with the standard inner product $\langle y,z\rangle:= y^{\top}z$. (In the subsequent of this example, in order to avoid confusion, whenever we write $X$, $Y$ we shall mean real symmetric matrices and not topological vector spaces, unless otherwise indicated.) Consider the linear operator $A:\mathcal{S}^{m}\rightarrow \mathbb{R}^{n}$ defined by $A(X):= (A_1\bullet X,A_2\bullet X,...,A_n\bullet X)^{\top}\in\mathbb{R}^{n}$ for some fixed $A_1,A_2,...,A_n\in\mathcal{S}^{m}$. The adjoint of $A$, $A^*:\mathbb{R}^{n}\rightarrow\mathcal{S}^{m}$, is defined by $A^*(y):= \sum_{j=1}^{n}y_jA_j$, for every $y\in\mathbb{R}^n$. This first class of semidefinite games treated is represented by games of the form $G=(S,T,u)$, where $u$ is given by:
\begin{equation*}
    u(X,y):= \langle y,A(X)\rangle=\sum_{j=1}^{n}y_jA_j\bullet X.
\end{equation*}
The strategy set $T$ of player $II$ is the standard n-simplex. Instead of player $I$ having the standard spectraplex as her strategy set (as in Example \ref{ex:4.4}), one can equip her with more general cone-leveled sets, such as
\begin{equation*}
    S:= \bigl\{X\in C: Tr(LX)\in[\lambda_1,\lambda_2]\bigr\},
\end{equation*}
for some $\lambda_1,\lambda_2>0$ and $L\in\mathcal{S}^m$, where $C$ is the convex cone of all positive semidefinite matrices $X$ that satisfy $\sum_{i=1}^{k}P_iXQ_i\in\mathcal{S}^{m}_{+}$ for some fixed (invertible) $m\times m$ matrices $P_i,\;Q_i$. The value of this game eventually boils down to the solution of a primal-dual pair of SDPs.\par
Despite the fact that this class of semidefinite games is not as popular as the next two in the relevant literature, it is strongly tied to the problem of verifying Slater's condition (as seen, for instance, in Example \ref{ex:4.4}), which is a key assumption that commonly appears in the theory of SDPs (e.g., see \cite{pataki2024exponential}, \cite[pp. 46, 76]{wolkowicz2012handbook}).\par

 For the second class of semidefinite games treated, let $X=W=\mathcal{S}^{m}$, $Y=Z=\mathcal{S}^n$, all equipped with the Frobenius inner product. The strategy sets of the players are the standard spectrahedra
\begin{equation*}
    S:= \left\{X\in \mathcal{S}^m_{+}: Tr(X)=1\right\},\;\; T:= \left\{Y\in \mathcal{S}^n_{+}: Tr(Y)=1\right\}.
\end{equation*}
These are bases of the cones of all positive semidefinite matrices $\mathcal{S}^{m}_{+}$ and $\mathcal{S}^{n}_{+}$, respectively. The payoff operator $A:\mathcal{S}^m\rightarrow\mathcal{S}^n$ is now a tensor, and the payoff function takes the form:
\begin{equation*}
    u(X,Y)=Y\bullet A(X)=\sum_{i,j,k,l}X_{ij}A_{ijkl}Y_{kl}
\end{equation*}
for every $X\in S,$ $Y\in T$, where the tensor $A$ satisfies the symmetric properties $A_{ijkl}=A_{jikl}$ and $A_{ijkl}=A_{ijlk}$ for every $i,j\in\{1,2,...,m\},\;k,l\in\{1,2,...,n\}$. The programs $(P_{I_{n}})$, $(D_{I_{m}})$ are once again SDP problems. Games of this form, which can be considered as natural generalizations of bimatrix games, have been studied very recently by Ickstadt et al. \cite{ickstadt2024semidefinite}.\par

The third class of semidefinite games embedded into our game model that is worth mentioning, is that of non-interactive zero-sum quantum games. In relation to the previous one, the (mixed) strategies of each player are now density matrices, belonging to the cone-leveled sets
\begin{equation*}
    S:= \{X\in \mathcal{S}^m_{+}(\mathcal{H}_1): Tr(X)=1\},\;\; T:= \{Y\in \mathcal{S}^n_{+}(\mathcal{H}_2): Tr(Y)=1\},
\end{equation*}
where $\mathcal{S}^m_{+}(\mathcal{H}_1)$ is the set of all $m\times m$ Hermitian positive semidefinite matrices acting on a finite-dimensional complex Euclidean space $\mathcal{H}_1$ (similarly for $\mathcal{S}^n_{+}(\mathcal{H}_2$) and $\mathcal{H}_2$). The payoff operator $A$ is now represented by a super-operator $\Phi:L(\mathcal{H}_1)\rightarrow L(\mathcal{H}_2)$ that preserves hermiticity, where $L(\mathcal{H}_i)$ is the space of all linear operators  $F:\mathcal{H}_i\rightarrow\mathcal{H}_i$, $i=1,2$. The reader who is interested in the different representations of the payoff super-operator $\Phi$, as well as in algorithms for solving non-interactive zero-sum quantum games via semidefinite programming, is referred to Jain and Watrous \cite{jain2009parallel} and the references within.

\vspace{-0.1cm} \subsection{A class of time-dependent games}\vspace{-0.2cm} 

This next class covers a vast number of time-dependent problems that can be interpreted as zero-sum games between two competitors, such as supply chain management, financial market dynamics or inventory and resource allocation problems in time-dependent transportation networks (see \cite{cachon2006game} for an extensive survey on different game-theoretic approaches of the aforementioned). We present a generic mathematical formulation for such time-dependent problems -- inspired by the bottleneck problem originally introduced by Bellman \cite{bellman1966dynamic} -- grounded in the cone-leveled setup of Section \ref{sec:3}. We then apply this model to a network defense problem against adversarial attacks, in the context of cybersecurity.\par

Let $X=L^{\infty}[0,T]^{m}$ and $W=L^1[0,T]^{m}$ paired with their weak and weak$^*$ topologies, respectively. Here, $L^{\infty}[0,T]$ represents the space of all essentially bounded measurable functions on $[0,T]$, and $L^1[0,T]$ represents the space of all Lebesgue integrable functions on $[0,T]$. Recall that $L^1[0,T]^*=L^{\infty}[0,T]$. The spaces $X,\;W$ are equipped with the bilinear form $\langle w,x\rangle:= \int_{0}^T w(t)^{\top}x(t)dt$. Accordingly, let $Y=L^{\infty}[0,T]^{n}$ and $Z=L^1[0,T]^{n}$ equipped with the bilinear form $\langle y,z\rangle:= \int_{0}^{T}y(t)^{\top}z(t)dt$. Let $A:X\rightarrow Z$ defined by 
\vspace{-0.2cm}\begin{equation*}
    (Ax)(t):= B(t)^{\top}x(t)-\int_{t}^{T}C(s,t)^{\top}x(s)ds,\;\; t\in[0,T],
\end{equation*}
and its adjoint $A^*:Y\rightarrow W$ by
\begin{equation*}
    (A^*y)(t):= B(t)y(t)-\int_{0}^{t}C(s,t)y(s)ds, \;\;t\in[0,T].
\end{equation*}
Here, $B(t)$ represents a $m\times n$ matrix for every $t\in[0,T]$, $C(s,t)$ a $m\times n$ matrix for every $s,t\in[0,T]$ which is equal to the zero matrix for $s>t$, and they are both continuous bounded Lebesgue measurable for every $s,t\in[0,T]$. Thus, $A$ is continuous with respect to the weak topology. Now consider the game $G=(S,T,u)$, where the payoff function is defined by 
\vspace{-0.1cm}\begin{equation*}
   u(x,y)=\langle y,Ax\rangle=\int_{0}^{T}x(t)^{\top}B(t)y(t)dt-\int_{0}^{T}\left(\int_{t}^{T}x(s)C(s,t)ds\cdot y(t)\right)dt. 
\end{equation*}
Notice that, due to the Fubini theorem, the payoff function is well defined and the relation $\langle y,Ax\rangle=\langle A^*y,x\rangle$ holds. Further, the strategy sets of the players are 
\begin{equation*}
  S:= \left\{x\in L^{\infty}_{+}[0,T]^{m}:\int_{0}^{T}\alpha(t)^{\top}x(t)dt\in[p_1,q_1]\right\},  
\end{equation*}
\begin{equation*}
    T:= \left\{y\in L^{\infty}_{+}[0,T]^{n}:\int_{0}^{T}\beta(t)^{\top}y(t)dt\in[p_2,q_2]\right\},\hspace{0.1cm}
\end{equation*}
for some $p_1,p_2,q_1,q_2>0$. The space $L^{\infty}_{+}[0,T]$ represents the space of all almost everywhere on $[0,T]$ non-negative bounded Lebesgue measurable functions. In addition, $\alpha\in C^{+}[0,T]^{m}$ and $\beta\in C^{+}[0,T]^{n}$, where $C^{+}[0,T]$ is the space of all non-negative continuous functions on $[0,T]$. These strategy sets are cone-leveled sets defined by the convex cones $C:= L^{\infty}_{+}[0,T]^{m}$ and $K:= L^{\infty}_{+}[0,T]^{n}$.\par

 For a concrete application of this time-continuous model, consider the following game: Player $I$ (network administrator) is searching for a continuous allocation plan of defensive resources over a network, while player $II$ (adversary) is aiming to compromise the network by launching attacks that degrade system performance or expose sensitive data. The strategy (vector) function $x(t)$ represents the allocation of $m$ defensive resources over time $t\in[0,T]$, while $y(t)$ represents the intensity of $n$ types of attacks launched on the network over the same time period. The vector $\alpha(t)$ can be thought of as a resource cost vector (e.g., $\alpha_i(t)=$ cost of deploying resource $i$ at time $t$), and the vector $\beta(t)$ can be similarly interpreted as the unit cost vector of each attack type at time $t$. The intervals $[p_1,q_1],\;[p_2,q_2]$ represent the (average) budget and adversarial demand restrictions over $[0,T]$ for each player, respectively. The matrix $B(t)$ represents effectiveness of defensive resources, i.e., $[B(t)]_{ij}=$ effectiveness of defensive mechanism $i$ against attack $j$ at time $t\in[0,T]$. In turn, the matrix $C(s,t)$ can be interpreted as a propagation matrix that represents the long-term effects of resource allocation. For example, assume that $[C(s,t)]_{ij}$ quantifies the impact of allocating resource 
$i$ at time $s$ on the effectiveness against attack $j$ at a later time $t$. Given the above payoff function, the network administrator is trying to enhance the immediate effectiveness of resource allocations against ongoing attacks and
mitigate the adverse delayed effects of past resource decisions, while the adversary's goal is the exact opposite.\par

 The -- associated with the game -- pair $\{(P_{\beta(t)}),(D_{\alpha(t)})\}$ represents a primal-dual pair of continuous linear programming problems. A number of duality results have long been given for such programs (Levinson \cite{levinson1966class}, Nash and Anderson \cite{nash1987linear}, Shapiro \cite{shapiro2001duality}), and algorithms are known for special cases, such as when $(P_{\beta(t)}),\;(D_{\alpha(t)})$ are separated CLPs (see Weiss \cite{weiss2008simplex}). A similar game analysis works for the reflexive Banach strategy spaces $L^p$, with $p\in(1,\infty)$. Also note that, while the cone of all almost everywhere non-negative functions has empty interior in $L^p$ for $p\in(1,\infty)$, it has non-empty interior for $p=\infty$.

\vspace{-0.1cm}\subsection{Polynomial games}\vspace{-0.2cm}

 Let $X=W=\mathbb{R}^{m+1}$ and $Y=Z=\mathbb{R}^{n+1}$, both equipped with the standard inner product $\langle \nu,\mu\rangle:= \nu^{\top}\mu$. We assume that $m,\;n$ are even positive integers. Also, let $A$ represent an $(n+1)\times (m+1)$ real matrix. The payoff function of the following game is defined by $u(\mu,\nu) := \mu^{\top}A^{\top}\nu$.\par
 
 We define the linear operator $\mathcal{H}:\mathbb{R}^{2k-1}\rightarrow\mathcal{S}^{k}$ whose image of every vector represents the associated Hankel matrix, for some $k\in\mathbb{N}$. That is, let
\begin{equation*}
    \mathcal{H}:\begin{pmatrix}
a_1\\
a_2\\
\vdots\\
a_{2k-1}
\end{pmatrix}\mapsto \begin{pmatrix}
a_1 & a_2 & \ldots & a_{k}\\
a_2 & a_3 & \ldots & a_{k+1}\\
\vdots & \vdots & \ddots & \vdots\\
a_{k} & a_{k+1} & \ldots & a_{2k-1}
\end{pmatrix}.
\end{equation*}
Its corresponding adjoint $\mathcal{H}^*:\mathcal{S}^{k}\rightarrow\mathbb{R}^{2k-1}$ is defined by adding all terms along each antidiagonal of a matrix, i.e., $\mathcal{H}^*(B)=(b_{11},2b_{12},b_{22}+2b_{13},...,2b_{k-1k},b_{kk})^\top$, where $B=(b_{ij})_{i,j=1}^{k}\in\mathcal{S}^k$.\par

The strategy set of player $I$ (player $II$ resp.) contains all $(m+1)$ ($n+1$ resp.) vectors of valid moments for some probability measure in the closed interval $[-1,1]$. By invoking classical results of moment spaces (see Parrilo \cite{parrilo2006polynomial} for a more precise analysis, as well as the references within), it can be shown that these sets take the following equivalent form:
\begin{equation*}
S:= \left\{\mu=(\mu_0,\mu_1,...,\mu_{m})\in\mathbb{R}^{m+1}:\mathcal{H}(\mu)\succeq 0,\;M_1^{\top}\mathcal{H}(\mu)M_1-M_2^{\top}\mathcal{H}(\mu)M_2\succeq 0,\;\alpha^{\top}\mu=1\right\},
\end{equation*}
\begin{equation*}
T:= \left\{\nu=(\nu_0,\nu_1,...,\nu_{n})\in\mathbb{R}^{n+1}:\mathcal{H}(\nu)\succeq 0,\;N_1^{\top}\mathcal{H}(\nu)N_1-N_2^{\top}\mathcal{H}(\nu)N_2\succeq 0,\;\beta^{\top}\nu=1\right\}.\;\;\;\;\;\;
\end{equation*}

Here $\alpha$ ($\beta$ resp.) is the $m+1$ ($n+1$ resp.) vector whose first entry is equal to one and all other entries are equal to zero, and
\begin{equation*}
  M_1:= \begin{pmatrix} 
I_{\frac{m}{2}}\\
0_{1\times\frac{m}{2}}\end{pmatrix},\;\; M_2:= \begin{pmatrix}
0_{1\times\frac{m}{2}}\\
I_{\frac{m}{2}}    
\end{pmatrix},\;\;N_1:= \begin{pmatrix} 
I_{\frac{n}{2}}\\
0_{1\times\frac{n}{2}}\end{pmatrix},\;\;N_2:= \begin{pmatrix}
0_{1\times\frac{n}{2}}\\
I_{\frac{n}{2}}    
\end{pmatrix}. 
\end{equation*}
 These strategy sets are basic cone-leveled defined by the convex cones
\begin{equation*}
C:= \left\{\mu=(\mu_0,\mu_1,...,\mu_{m})\in\mathbb{R}^{m+1}:\mathcal{H}(\mu)\succeq 0,\;M_1^{\top}\mathcal{H}(\mu)M_1-M_2^{\top}\mathcal{H}(\mu)M_2\succeq 0\right\},
\end{equation*}
\begin{equation*}
K:= \left\{\nu=(\nu_0,\nu_1,...,\nu_{n})\in\mathbb{R}^{n+1}:\mathcal{H}(\nu)\succeq 0,\;N_1^{\top}\mathcal{H}(\nu)N_1-N_2^{\top}\mathcal{H}(\nu)N_2\succeq 0\right\}.\;\;\;\;\hspace{0.23cm}
\end{equation*}
The game structure just described represents the ``translation" of Parrilo's \cite{parrilo2006polynomial} work into the cone-leveled game model introduced here. In particular, Parrilo \cite{parrilo2006polynomial} showed that this two-player zero-sum game constitutes the mixed extension of a polynomial game, as its original normal form is described by the payoff function $P(x,y):= \sum_{i=0}^{m}\sum_{j=0}^{n}p_{ij}x^{i}y^{j}$, for $x,y\in[-1,1]$ (or any compact interval by a simple linear transformation) and $p_{ij}=a_{ji}$ for every $i,j$, and that the associated pair \hyperlink{newprimaldual}{$\{(P_\beta),(D_\alpha)\}$} can efficiently be reduced to a pair of SDP problems. Polynomial and ``polynomial-like" games have attracted significant interest since the early days of game theory (see \cite{dresher1950polynomial}), and they are strongly tied to concern (5) of Kuhn and Tucker in the preface of \cite{kuhn1953contributions}. 

\vspace{-0.1cm}\subsection{Homogeneous Separable games}\vspace{-0.2cm}

\noindent From the discussion in Subsection \ref{sec:3.3} we can see that any two-player zero-sum (positive) homogeneous separable game $G=(S,T,u)$ with cone-leveled strategy sets $S,\;T$, defined over finite-dimensional Euclidean spaces $X,\;Y$, can be reduced to a pair of conic convex programs. In such games, the payoff function is of the form
\begin{equation*}
    u(x,y)=\sum_{i=1}^{n}f_i(x)g_i(y),\;\;\forall x\in S,\;\forall y\in T,
\end{equation*}
where every $f_i$ is concave and $q$-homogeneous for some $q\neq0$, and every $g_i$ is convex and $p$-homogeneous for some $p\neq0$. Depending on the form of the associated pair \hyperlink{newprimaldual}{$\{(P_\beta),(D_\alpha)\}$} of convex problems and on other characteristics of the functions $f_i,\;g_i$ (e.g., strong convexity, $L$--smoothness, etc.), one can apply appropriate algorithms (see, for instance, \cite{nesterov1994interior, nocedal1999numerical}).

\vspace{-0.1cm}\section{Extensions to more general strategy sets}\label{sec:6}\vspace{-0.2cm}

Here we extend Theorem \ref{th:3.1} to two-player zero-sum games with strategy sets that are finite intersections and finite unions of cone-leveled sets. The former represent strategies that satisfy finite systems of linear (in)equalities, and the latter have a structure sufficient for the existence and computation of subgame Nash equilibria. Consider the cone-leveled sets
\begin{equation}\label{eq:6.1}
    S_i=\{x\in C_i:\langle\alpha_i,x\rangle\in H_i\},\quad i\in\mathcal{I},\hspace{0.22cm}
\end{equation}
\begin{equation}\label{eq:6.2}
    T_j=\{y\in K_j:\langle y,\beta_j\rangle\in Q_j\},\quad j\in\mathcal{J},
\end{equation}

\noindent where $\{C_i\}_{i\in\mathcal{I}}$, $\{K_j\}_{j\in\mathcal{J}}$ are families of convex cones and, without loss of generality, $Q_j\subseteq [1,\delta_j]$ $\forall j\in\mathcal{J}$ and $H_i\subseteq[\gamma_i,1]$ $\forall i\in\mathcal{I}$, where $\delta_j\geq1\;\forall j\in\mathcal{J},\;\gamma_i\in(0,1]\;\forall i\in\mathcal{I}$. Hereafter, we will assume that $+\infty>|\mathcal{I}|=m\geq 2$ and $+\infty>|\mathcal{J}|=n\geq 2$.

\vspace{-0.1cm}\subsection{Intersections of cone-leveled strategy sets}\label{sec:6.1}\vspace{-0.2cm}

Consider the convex cones $C:= \bigcap_{i=1}^{m}C_i$, $K:= \bigcap_{j=1}^{n}K_j$ and the pair of conic programs $\{(\widehat{P}_{\beta_j}),(\widehat{D}_{\alpha_{i}})\}$ for some $(i,j)\in\mathcal{I}\times\mathcal{J}$. Here, $(\widehat{P}_{\beta_j})$ is the program \hyperlink{primalbeta}{$(P_{\beta_{j}})$} with objective function $\langle\alpha_i,x\rangle$ and the additional set of constraints $\{\langle\alpha_k,x\rangle=\langle\alpha_l,x\rangle\;|\;k\neq l\in\mathcal{I}\}$. Accordingly, $(\widehat{D}_{\alpha_{i}})$ is the program \hyperlink{dualalpha}{$(D_{\alpha_{i}})$} with objective function $\langle y,\beta_j\rangle$ and additional set of constraints $\{\langle y,\beta_s\rangle=\langle y,\beta_t\rangle\;|\;s\neq t\in\mathcal{J}\}$. Note that this does not necessarily represent a primal-dual pair. The extension of Theorem \ref{th:3.1} to finite intersections -- given a zero duality gap between the aforementioned pair -- is natural and identical methods are used, hence certain parts of the proof of the following result will be omitted. In order to remain consistent with Definition \ref{def:2.2}, we will assume, without loss of generality, that whenever $\bigcap_{i=1}^{m}S_i\neq\emptyset$, there exists $x\in\bigcap_{i=1}^{m}S_i$ such that $\langle\alpha_i,x\rangle=1$ for every $i\in\mathcal{I}$. We make a similar assumption for the set $\bigcap_{j=1}^{n}T_j$.\quad\vspace{0.2cm}

\begin{proposition}\label{prop:6.1}
    Consider a two-player zero-sum game $G=(S,T,u)$ defined by \hyperlink{f2}{(F2)} and \hyperlink{f3}{(F3)}, where $S=\bigcap_{i=1}^{m}S_i$, $T=\bigcap_{j=1}^{n}T_j$, and suppose that $\underline{v}>0$. If the pair $\{(\widehat{P}_{\beta_{j_0}}),(\widehat{D}_{\alpha_{i_0}})\}$ is feasible for some $(i_0,j_0)\in\mathcal{I}\times\mathcal{J}$ and $\text{val}(\widehat{P}_{\beta_{j_0}})=\text{val}(\widehat{D}_{\alpha_{i_0}})>0$, then $\underline{v}=\overline{v}$, and $v=\frac{1}{\text{val}(\widehat{P}_{\beta_{j_0}})}$. Moreover, if $x^*$ is an optimal solution for $(\widehat{P}_{\beta_{j_0}})$ and $y^*$ is an optimal solution for $(\widehat{D}_{\alpha_{i_0}})$, then $(vx^*,vy^*)$ is a saddle point.\quad\vspace{0.2cm}
\end{proposition}
\begin{proof}
    Consider the conic programs \hyperlink{phat}{$(\widehat{P}\hspace{0.05cm})$}, \hyperlink{phati0}{$(\widehat{P}_{i_0})$} of Lemma \ref{lem:3.2}. These are equivalent by the latter lemma and, given that one of them is feasible, the relation $\underline{v}\geq \frac{1}{\text{val}(\widehat{P}_{i_0})}$ holds due to Lemma \ref{lem:securitylevels}. In turn, let the program $(\overline{P}_{\beta_{j_0}})$ be the restriction of $(\widehat{P}_{\beta_{j_0}})$ defined by the restricted constraint ``$x\in C\setminus \left(\bigcap_{i=1}^{m}\mathcal{D}_i\right)$". Then, $\text{val}(\widehat{P}_{i_0})\leq \text{val}(\overline{P}_{\beta_{j_0}})$. Moreover, since $(\widehat{P}_{\beta_{j_0}})$ is feasible with positive value, $(\overline{P}_{\beta_{j_0}})$ is also feasible and $\text{val}(\overline{P}_{\beta_{j_0}})=\text{val}(\widehat{P}_{\beta_{j_0}})$, hence $\underline{v}\geq\frac{1}{\text{val}(\widehat{P}_{\beta_{j_0}})}$. By avoiding unnecessary repetitions of similar methods that appear in Theorem \ref{th:3.1} and Lemma \ref{lem:3.2} for the dual counterpart,  the inequality $\overline{v}\leq\frac{1}{\text{val}(\widehat{D}_{\alpha_{i_0}})}$ also holds. By (\ref{eq:2.5}) we get $\underline{v}=\overline{v}$ and $v=\frac{1}{\text{val}(\widehat{P}_{\beta_{j_0}})}$. By invoking Lemmas \ref{lem:securitylevels} and \ref{lem:3.2}, one can show that if $x^*$ is an optimal solution for $(\widehat{P}_{\beta_{j_0}})$, then $vx^*$ is an optimal solution for \hyperlink{phat}{$(\widehat{P}\hspace{0.05cm})$}, thus an optimal strategy for player $I$, as in the proof of Theorem \ref{th:3.1}. Accordingly, if $y^*$ is an optimal solution for $(\widehat{D}_{\alpha_{i_0}})$, then $vy^*$ is an optimal strategy for player $II$, hence $(vx^*,vy^*)$ is a Nash equilibrium.
\end{proof}

\begin{remark} As already mentioned, the pair $\{(\widehat{P}_{\beta_{j_0}}),(\widehat{D}_{\alpha_{i_0}})\}$ is not necessarily a primal-dual pair. As a result, certificates for strong duality, like Proposition \ref{prop:2.1}, cannot be applied. The absence of a duality gap and the existence of optimal solutions boils down to other conditions. For example, if there exist $x^*\in\mathcal{F}(\widehat{P}_{\beta_{j_0}})$, $y^*\in\mathcal{F}(\widehat{D}_{\alpha_{i_0}})$, 
  $(s_1,s_2,...,s_m)\in C_1^*\times C_2^*\times...\times C_m^*$ and $(t_1,t_2,...,t_n)\in K_1^*\times K_2^*\times...\times K_n^*$ such that $Ax^*-\beta_{j_0}=t_1+t_2+...+t_n$, $-A^*y^*+\alpha_{i_0}=s_1+s_2+...+s_m$, $\langle s_k,x^*\rangle=0$ $\forall k\in\mathcal{I}$, and $\langle y^*,t_l\rangle=0$ $\forall l\in\mathcal{J}$, then $\text{val}(\widehat{P}_{\beta_{j_0}})=\text{val}(\widehat{D}_{\alpha_{i_0}})$ and $x^*,y^*$ are optimal solutions (complementary slackness).
  \end{remark}
  
\vspace{-0.1cm}\subsection{Unions of cone-leveled strategy sets}\label{sec:6.2}\vspace{-0.2cm}

When the strategy sets are finite unions of cone-leveled sets, then checking for the existence of subgame Nash equilibria is straightforward: One simply has to restrict to some subgame $G_{ij}=(S_i,T_j,u)$ and verify the assumptions of Theorem \ref{th:3.1}. Of course, if $S,\;T$ are finite unions of bases of convex cones, then every $G_{ij}$ is guaranteed to have a saddle point by Theorem \ref{th:3.3}. In the next proposition we give sufficient conditions for a subgame Nash equilibrium to be a global one in the former, more general case.\par

 First, let $S:=\bigcup_{i=1}^{m}S_i$, $T:= \bigcup_{j=1}^{n}T_j$, where $S_i,\;T_j$ are as in (\ref{eq:6.1}), (\ref{eq:6.2}). We define the program $(P_{i\cdot}')$ as the restriction of \hyperlink{p1}{$(P1)$} described by the restricted constraint ``$x\in S_i$", for $i\in\mathcal{I}$. Accordingly, let $(D_{\cdot j}')$ be the restriction of \hyperlink{d1}{$(D1)$} described by the restricted constraint $y\in T_j$ for $j\in\mathcal{J}$. Then $\text{val}(P_{i\cdot}')\leq \text{val}(P1)=\underline{v}$ and $\text{val}(D_{\cdot j}')\geq \text{val}(D1)=\overline{v}$ by Lemma \ref{lem:securitylevels}. In association with the above, let the pair $\{(P_{ij}'),(D_{ij}')\}$, where $(P_{ij}')$ is identical to the problem $(P_{i\cdot}')$ but $T$ is replaced by $T_j$. Accordingly, $(D_{ij}')$ is defined similarly to $(D_{\cdot j}')$, but $S$ is replaced by $S_i$. Lastly, consider the following primal-dual pair of conic linear programs:
\hypertarget{primaldualij}{}\hypertarget{pij}{} \hypertarget{dij}{}
\begin{center}
\vspace{-0.7cm}\begin{equation*}{(P_{ij})\;\;\;\;\;\;\;\;}
\begin{array}{ll}
    \inf\limits_{x}\;\; \langle \alpha_i,x\rangle\\
    \text{s.t.}\;\; Ax-\beta_j\in K_j^*\\ \;\;\;\;\;\;\;x\in C_i
    \end{array}
\end{equation*}
\end{center}
\begin{center}
\vspace{-0.7cm}\begin{equation*}{\;\;\;\;\;(D_{ij})\;\;\;\;\;\;\;\;}
\begin{array}{ll}
    \sup\limits_{y}\;\; \langle y,\beta_j\rangle\\
    \text{s.t.}\;\; -A^*y+\alpha_i\in C_i^*\\ \;\;\;\;\;\;\;\;y\in K_j
    \end{array}
    \end{equation*}
\end{center}
\begin{proposition}\label{prop:6.2}
    Consider a two-player zero-sum game $G=(S,T,u)$ defined by \hyperlink{f2}{(F2)} and \hyperlink{f3}{(F3)}, with strategy sets $S,\;T$ that are unions of cone-leveled sets, as above, and suppose that $\underline{v}>0$. If the pair \hyperlink{primaldualij}{$\{(P_{ij}),(D_{ij})\}$} is feasible for some $(i,j)\in\mathcal{I}\times\mathcal{J}$ such that $\text{val}(P_{ij})=\text{val}(D_{ij})>0$, $\text{val}(P_{ij}')=\text{val}(P_{i\cdot}')$ and $\text{val}(D_{ij}')=\text{val}(D_{\cdot j}')$, then $\underline{v}=\overline{v}$, and $v=\frac{1}{\text{val}(P_{ij})}$. Moreover, if $x^*$ is an optimal solution for \hyperlink{pij}{$(P_{ij}$)} and $y^*$ is an optimal solution for \hyperlink{dij}{$(D_{ij})$}, then $(vx^*,vy^*)$ is a saddle point. 
\end{proposition}

\begin{proof}
    Take the two-player zero-sum subgame $G_{ij}=(S_i,T_j,u)$. Theorem \ref{th:3.1} yields 
    \begin{equation}\label{eq:6.3}
    \text{val}(P_{ij}')=\text{val}(D_{ij}')=\displaystyle\frac{1}{\text{val}(P_{ij})}.
    \end{equation}
    Now $\text{val}(P_{ij}')=\text{val}(P_{i\cdot}')\leq\underline{v}$ and $\text{val}(D_{ij}')=\text{val}(D_{\cdot j}')\geq\overline{v}$. By (\ref{eq:6.3}) and (\ref{eq:2.5}) we finally get 
    \begin{equation*}
        \underline{v}=\overline{v}=\displaystyle\frac{1}{\text{val}(P_{ij})}.
        \end{equation*}
    Additionally, if $x^*$ is an optimal solution for \hyperlink{pij}{$(P_{ij})$}, then $\text{val}(P_{ij}')x^*$ is an optimal solution for $(P_{ij}')$ by Theorem \ref{th:3.1}. It follows that $\text{val}(P_{ij}')x^*=vx^*$ is an optimal solution for \hyperlink{p1}{$(P1)$}, that is, it is an optimal strategy for player $I$ (due to Lemma \ref{lem:securitylevels}). Similarly, if $y^*$ is an optimal solution for \hyperlink{dij}{$(D_{ij})$}, then $vy^*$ is an optimal strategy for player $II$, and $(vx^*,vy^*)$ is a saddle point.
\end{proof}

\vspace{-0.1cm}\section{Future directions}\label{sec:7}\vspace{-0.2cm}

After proving that every two-player zero-sum game with cone-leveled strategy sets can be described by a primal-dual pair of conic linear programs, the following question naturally arises: \textit{When is a two-player zero-sum game equilibrium-equivalent to a game with cone-leveled strategy sets?} We believe that an answer to this question will offer a significant contribution towards the associated mathematical concern raised by Kuhn and Tucker \cite{kuhn1953contributions}, regarding the existence of ``\textit{a computational technique of general
applicability for zero-sum two-person games with infinite sets of strategies}". The translation of Parrilo's work on polynomial games \cite{parrilo2006polynomial} into our framework indicates the existence of strategy sets that can be written in the cone-leveled form, even though their original form does not directly imply so. Therefore, a deeper search regarding equivalent strategic forms of known two-player zero-sum games, for which the existence and computation of Nash equilibria is of high interest, is encouraged.\par

In addition, the almost equivalence provides a, perhaps surprising, new insight regarding the determination of strict feasibility. For semidefinite programming in particular, the decision of feasibility in polynomial time is a fundamental open problem (Pataki and Touzov \cite{pataki2024exponential}). The following is a corollary of Theorem \ref{th:4.1}: \textit{Determining strict feasibility for a pair of SDPs in polynomial time is equivalent to either constructing a semidefinite game in polynomial time with non-zero value, or constructing one in polynomial time so that a zero game value and the infeasibility of a certain null set are guaranteed}. Here, by ``construction" of a semidefinite game in polynomial time we mean the construction of a base of the non-negative orthant, and a base of the cone of positive semidefinite matrices, in polynomial time. This equivalent formulation signifies a novel approach to solving this open problem.\par

The analysis of our work is also related to the theory of zero-sum polymatrix games, which are multiplayer generalizations of two-player zero-sum games defined by networks \cite{cai2011minmax}. Cai et al. \cite{cai2016zero} proved that such games have Nash equilibria which can be efficiently computed by a certain linear program. In a more recent paper, Ickstadt et al. \cite{ickstadt2026semidefinite} generalized this concept and defined network semidefinite games, with the aim of extending models of quantum strategic interactions to multiplayer games. In particular, they showed that such games can be solved by an SDP, whose form is very similar to the LP that appears in \cite{cai2016zero}. Although these zero-sum multiplayer games are defined over graphs, their payoff functions are still linear (defined by specific linear operators) and the strategy sets of the players represent cone-leveled sets (in fact bases of convex cones). Further, the main existential results on Nash equilibria in both works rely heavily on duality theory. Another natural direction is therefore associated with the extension of the definition of polymatrix games to a more general strategic setup characterized by \hyperlink{f1}{(F1)}--\hyperlink{f3}{(F3)}, and with the use of strong duality to show that the existence and calculation of Nash equilibria for such games boils down to the solution of a specific conic linear program.

\vspace{-0.3cm}\section*{Acknowledgements}\vspace{-0.3cm}

I would like to thank Eilon Solan for his helpful comments and suggestions on an early draft of this paper, as well as Alexander Shapiro for helping me develop a better understanding concerning some important parts of his work \cite{shapiro2001duality}. I am very grateful to G\'abor Pataki for introducing me to some deeper concepts of conic linear programming related to this work. I would also like to thank Panayotis Mertikopoulos for pointing out \cite{ickstadt2024semidefinite}. Special thanks to Panagiotis Andreou and Pavlos Zoubouloglou, whose suggestions led to great improvements of this paper.

\vspace{-0.5cm}\bibliography{references}

\end{document}